\documentclass[12pt]{amsart}

\usepackage{cleveref}
\usepackage{mathabx}
\usepackage{tikz}

\newtheorem{thm}{Theorem}[section]
\newtheorem{cor}[thm]{Corollary}
\newtheorem{lem}[thm]{Lemma}
\newtheorem{prop}[thm]{Proposition}
\theoremstyle{definition}
\newtheorem{defn}[thm]{Definition}
\theoremstyle{remark}
\newtheorem{rem}[thm]{Remark}
\newtheorem*{ex}{Example}
\numberwithin{equation}{section}

\newcommand{\noi}{\noindent}
\newcommand{\ra}{\rightarrow}
\newcommand{\hk}{\widehat{K}}
\newcommand{\hf}{\widehat{f}}

\begin{document}
	
	%
	%
	%
	%
	%
	%
	%
	%
	%

	\title[Equality in Hausdorff-Young for Hypergroups]{Equality in Hausdorff-Young  for Hypergroups}

	\author[C. Bandyopadhyay]{Choiti Bandyopadhyay}
	
	\address{%
		Department of Mathematics \& Statistics, Indian Institute of Technology Kanpur, India 208016.}
	
	\email{choiti@ualberta.ca}
	
	\author[P. Mohanty]{Parasar Mohanty}
	\address{Department of Mathematics \& Statistics, Indian Institute of Technology Kanpur, India 208016.}
	\email{parasar@iitk.ac.in}
	
	
	
	\begin{abstract}
		It was shown in \cite{SD} that one can extend the domain of Fourier transform of a commutative hypergroup $K$ to $L^p(K)$ for $1\leq p \leq 2$, and the Hausdorff-Young inequality holds true for these cases. In this article, we examine the structure of non-zero functions in $L^p(K)$ for which equality is attained in the Hausdorff-Young inequality, for $1<p<2$, and further provide a characterization for the basic uncertainty principle for commutative hypergroups with non-trivial centre.
	\end{abstract}
	
	\maketitle

	
	\section{Introduction}
	
	In the classical setting, the  basic uncertainty principle simply states that a function $f$ and its Fourier transform $\widehat{f}$ cannot both be simultaneously concentrated on `small' sets. From elementary complex analysis, we know that if a non-trivial integrable function $f$ on $\mathbb{R}$ has compact support, then its Fourier transform $\widehat{f}$ can be extended to an entire function, and therefore  cannot have compact support. Hence roughly speaking, one can conclude that  for a function $f$ on $\mathbb{R}$, both $f$ and $\widehat{f}$  cannot be compactly supported.

	The situation changes when we address this problem in the setting of locally compact abelian groups. For example, if a commutative group $G$ is compact or discrete, then one can find continuous functions $f$ on $G$ such that both $f$ and $\widehat{f}$ would have compact support. Define the set
	\begin{eqnarray*}
		PW &:=&  \{ G : G \mbox{ is a locally compact abelian group, and there does not}\\
		& &  \mbox{ exist a non-trivial function } f \mbox{ on } G \mbox{ such that both } f \mbox{ and } \widehat{f}\\
		& & \mbox{ have compact support} \}.
	\end{eqnarray*}
	We say that a locally compact abelian group $G$ abides by the  basic uncertainty principle if $G\in PW$. As mentioned before, if an abelian group $G$ is compact or discrete, then $G\notin PW$. For a complete characterization of abelian groups which obey the { basic} uncertainty principle, one can look into the criteria for which equality is attained in the Hausforff-Young inequality: 
	\begin{equation}
	||\widehat{f}||_q \leq ||f||_p,\;\;{\rm for}\; f\in L^p(G), \; 1< p <2\; \mbox{  \ and \ }\; 1/p + 1/q=1.
	\label{HYI}
	\end{equation}  It turns out \cite{HH,HR2,LV}  that a complete characterization is possible in terms of the best constant achieved in the Hausdorff-Young inequality on $L^p(G)$, $1<p<2$.

	Next, it will be of interest to consider this problem in a more general setting of hypergroups. Roughly speaking, a hypergroup is a locally compact Hausdorff space with an associative measure algebra that admits an identity and a topological involution. As one would expect, a hypergroup can also be naturally perceived as a certain generalization of a locally compact group, where the product of two points may be a certain compact set, instead of a single point. In this sense, the category of locally compact groups is the most trivial example of the category of hypergroups. 
	
	The concept of hypergroups arises naturally in abstract harmonic analysis in terms of double-coset spaces and orbit spaces of certain affine actions of locally compact groups, solution spaces of partial differential equations and orthogonal polynomials in one and several variables, which appear frequently in different areas of research including Lie groups, double-coset spaces, dynamical systems and ordinary and partial differential equations, to name a few. 
	
	The fact that hypergroups contain more generalized objects arising from different fields of research than the classical group theory and yet sustains enough structure to allow an independent theory to develop, makes it an intriguing and useful area of study with essentially a broader area of applications. We define:
	\begin{eqnarray*}
		PW_H &:=&  \{ K : K \mbox{ is a commutative hypergroup, and there does not exist}\\
		& & \mbox{ a non-trivial function } f \mbox{ on } K \mbox{ such that both } f \mbox{ and } \widehat{f}\\
		& & \mbox{ have compact support} \}.
	\end{eqnarray*}
	We say that a commutative hypergroup $K$ abides by the  basic uncertainty principle if $K\in PW_H$. Note that the Hausdorff-Young inequality also holds true for hypergroups \cite{BH,JE,SD}. Hence it is of interest to investigate when and why the Hausdorff-Young inequality becomes sharp for a commutative hypergroup, and how it relates to the  basic uncertainty principle for the respective hypergroup.

	Equality in \Cref{HYI} was characterized  for locally compact abelian groups in \cite{HH}. In a follow-up paper \cite{H2}, this issue was addressed  for compact groups as well. In the Euclidean setting, Babenko (when $q$ is an even integer) \cite{B} and later Beckner \cite{Be} (for $1<p<2$)  found out the best constant in \Cref{HYI} for $G=\mathbb R^n$.  Russo \cite{R} further extended the result of Hewitt and Hirschman \cite{HH} to unimodular locally compact groups. In a recent paper \cite{CMP}, Cowling et.al. studied the best constant for certain  non-abelian groups as well. 
	
	The main goal of this article is to investigate when equality is attained in the Hausdorff-Young inequality for a commutative hypergroup, and thereby provide a certain characterization of commutative hypergroups which obey the basic uncertainty principle, in terms of the best constant attained in \Cref{HYI}. The rest of this article is organized as the following.
	
	In the next, \textit{i.e,} second section, we list and briefly discuss some of the basic notations, definitions and well-known facts regarding (topological) hypergroups that we will use throughout the article. 
	
	The third section is dedicated to the main results of the article. In this section, we first recall some well-known results that we would need for the rest of the article. We then proceed towards proving the first main result which states that for any commutative hypergroup $K$, if a non-trivial function $f\in L^p(K)$ attains equality in the Hausdorff-Young inequality for some $p\in (1, 2)$, then the supports of both $f$ and $\widehat{f}$ are compact open subsets of $K$ and $\hk$ respectively, \textit{i.e.} in particular, $K\notin PW_H$. This further provides us with useful insights into the specific structure of such an $f$, and specify where and why this theory deviates from the classical theory of locally compact abelian groups. With this in mind, we test equality in Hausdorff-Young with the first obvious candidate for such an $f$, which is a certain characteristic function. 
	
	We see that Fourier transforms of such characteristic functions on $K$ are supported on annihilators, which forces both the function and its Fourier transform to have compact support, and in turn equality is indeed attained in Hausdorff-Young. On the other hand, using this information, we see that if a commutative hypergroup $K$ with non-trivial centre sits outside $PW_H$, then one can always find a non-trivial function $f$ on $K$ such that $f$ attains equality in $L^p(K)$ for each $p\in [1, 2]$. 
	
	In fact, we deduce that given any commutative hypergroup $K\in PW_H^c$, one can always construct an open subhypergroup $H$ of $K$ such that the family of translates of some fixed non-zero $f \in C_c(K)$ by certain measures on $H$ is very small in $L^2(K)$, and has a finite dimensional linear span. Finally, this provides us with a certain characterization in the class of all commutative hypergroups with non-trivial centre (with respect to a Haar measure). We see that such a hypergroup will not follow the basic uncertainty principle if and only if it admits a non-trivial function $f$ such that $||\widehat{f}||_q=||f||_p$ for some $p\in (1, 2)$, $1/p + 1/q=1$, \textit{i.e,} $1$ is the best constant of Hausdorff-Young inequality for some $p\in (1, 2)$. We can immediately infer that hypergroups that are either compact or discrete or certain joins, will not follow the basic uncertainty principle.

	\section{Preliminary}
	
	
	Here we first list a basic set of notations, followed by brief formal definitions of the basic objects and techniques used throughout the article.
	

	For any locally compact Hausdorff topological space $X$, we denote by $M(X)$ the space of all regular complex Borel measures on $X$, where $M^+(X)$ denote the subset of $M(X)$ consisting of all finite non-negative, regular Borel measures  on $X$. Moreover, $\mathcal{B}(X), C(X), C_c(X)$ and $C_0(X)$ respectively denote the function spaces of all Borel measurable functions, bounded continuous functions, compactly supported continuous functions and continuous functions vanishing at infinity on $X$.  The support of a Borel measure $\mu$ on $X$ is denoted by $supp(\mu)$. If $f\in L^p(X, \omega)$, where $\omega$ is a regular positive Borel measure on $X$, then  $supp(f)$ simply denotes the essential support of $f$.

	We follow convention and denote the unit circle and the closed unit disc in $\mathbb{C}$ by $\mathbb{T}$ and $\mathbb{D}$ respectively.  Recall that the signum function $sgn:\mathbb{C}\setminus \{0\}\ra \mathbb{T}$ is defined as $sgn(z) = z/|z|$. We set $sgn(0)=0$ as per convention. For any subset $A\subseteq X$, the characteristic function associated to $A$ is denoted by $\chi_{_A}$. A function $F$ on $X$ is called a simple function if there exists mutually disjoint Borel sets of finite measure $E_1, E_2, \ldots , E_n \subseteq X$   and scalars $a_i, b_j \in \mathbb{R}$, $1\leq i, j \leq n$ for some $n\in \mathbb{N}$, such that
	$F(x) = \sum_{k=1}^n
	\exp({a_k + i b_k}) \chi_{_{E_k}}(x)$
	for each $x\in X$.
	
	Unless mentioned otherwise, the space $M^+(X)$ is equipped with the \textit{cone topology} \cite{JE}, \textit{i.e,} the weak topology on $M^+(X)$ induced by $ C_c^+(X)\cup \{\chi_{_X}\}$. We denote  the set of all compact subsets of $X$ by $\mathfrak{C}(X)$, and consider the \textit{Michael topology}\cite{MT} on it, which makes it into a locally compact Hausdorff space.  For any element $x\in X$, we denote by $p_x$ the point-mass measure or the Dirac measure at the point $\{x\}$.

	Note that in terms of the definitions and notations, we follow Jewett's notion of hypergroups, which he called `convos' in \cite{JE}. As discussed in the introduction, a hypergroup is a pair $(K,*)$ where $*$ is a convolution on $M(K)$ such that $(M(K),*)$ assumes a certain algebraic structure with identity and involutions. It can be shown that a vast class of double-coset spaces and orbit spaces of locally compact groups, polynomial-spaces in several variables and certain subsets of $\mathbb{R}$ naturally admit a hypergroup structure on their respective measure spaces (see \cite{CB1,CB2,BH,JE, RO} for detailed examples).

	\begin{defn}\textbf{(Hypergroup)} A pair $(K, *)$ is called a (topological) hypergroup if it satisfies the following properties :
		
		\begin{description}
			\item[(A1)] $K$ is a locally compact Hausdorff space and $*$ defines a binary operation on $M(K)$ such that $(M(K), *)$ becomes an associative algebra.
			
			\item[(A2)] The bilinear mapping $* : M(K) \times M(K) \rightarrow M(K)$ is positive continuous.
			
			\item[(A3)] For any $x, y \in K$ the measure $p_x * p_y$ is a probability measure with compact support.
			
			\item[(A4)] The map $(x, y) \mapsto \mbox{ supp}(p_x * p_y)$ from $K\times K$ into $\mathfrak{C}(K)$ is continuous.
			\item[(A5)] There exists an element $e \in K$ such that $p_x * p_e = p_e * p_x = p_x$ for any $x\in K$.
			\item[(A6)] There exists an involution $x\mapsto \widetilde{x}$ on $K$ such that $e \in \mbox{supp} (p_x * p_y)$ if and only if $x=\widetilde{y}$.
		\end{description}
	\end{defn}
	
	The element $e$ in the above definition is called the \textit{identity} of $K$. Both the identity and  involution map are necessarily unique \cite{JE}. Recall that a topological involution $x\mapsto \widetilde{x}$ on $K$ is called an involution for hypergroup if for any $\mu, \nu \in M(K)$ we have that
	$$(\mu * \nu)\widetilde{} = \widetilde{\nu} * \widetilde{\mu} ,$$
	where for any measure $\omega \in M(K)$ we have that $\widetilde{\omega}(B) := \omega(\widetilde{B})$ for any Borel measurable subset $B$ of $K$. For any $f\in \mathcal{B}(K)$, the function $\widetilde{f}$ is defined as $\widetilde{f}(x):=f(\widetilde{x})$ for each $x\in K$. Hence it follows immediately   that $\int_K\widetilde{f} \ d\mu = \int_K f \ d\widetilde{\mu}$ for any $\mu\in M(K)$. For any two subsets $A,B \subset K$ the convolution is defined as the following:
	$$A*B := \bigcup_{x\in A, y\in B}  supp(p_x*p_y) .$$
	Note that unlike locally compact group, the convolution of two connected sets need not be connected \cite{BH, JE}.
	
	A non-empty closed subset $H$ of a hypergroup $K$ is called a \textit{subhypergroup} if $H*H\subseteq H$ and $\widetilde{H}=H$. Note that the Michael topology on $\mathfrak{C}(H)$ is same as the relative topology induced as a subset of $\mathfrak{C}(K)$ \cite{JE}. The \textit{centre} $Z(K)$ of a hypergroup $K$ is defined as the following:
	$$Z(K) := \{x\in K : p_{_x} * p_{_{\widetilde{x}}} = p_{_{\widetilde{x}}} * p_{_x} = p_{_e}\}.$$
	We see immediately that $Z(K)$ is a subhypergroup of $K$, and is   a  commutative locally compact group itself, which is also known as the \textit{maximum subgroup} of $K$. It is easy to see that this definition coincides with the definition of centre for semihypergroups \cite{CB2}, \textit{i.e,} for any hypergroup $K$, we have$$Z(K) = \{x\in K : supp(p_{_x}*p_{_y}), supp(p_{_y}*p_{_x}) \mbox{ are singleton sets for each } y\in K\}.$$
	
	Given any subset $E$ of a hypergroup $K$, the subhypergroup generated by $E$, denoted as $\langle E \rangle$, is the smallest subhypergroup of $K$ that contains $E$. The following proposition can be proved in a straight-forward fashion.

	\begin{prop} \label{subgen}
		Let $K$ be a hypergroup and $E$ be a  neighborhood of $e$ in $K$. Then  $\langle E\rangle$ is essentially the closure of the following set:
		\begin{equation*}
		\underset{\underset{a_i\in E \cup \widetilde{E}, 1\leq i \leq n} {n\in \mathbb{N} }} {\bigcup} \Big{(} \{a_1\}*\{a_2\}*\cdots * \{a_n\} \Big{)} .
		\end{equation*} 
	\end{prop}
	We say that a neighborhood $E$ of $e$ is symmetric if $\widetilde{E}=E$. For any symmetric neighborhood $E$ of $e$, we denote by $M_0(\langle E\rangle)$ the subset of $M(\langle E\rangle)$ consisting of convolutions of finitely many measures in $M(E)$. In particular,
	$$M_0(\langle E\rangle) := \{\mu_1*\mu_2*\cdots *\mu_n : \mu_i\in M(E) \mbox{  \ for  \ } 1\leq i \leq n, \ n\in \mathbb{N}\}.$$
	
	Now for any $f\in \mathcal{B}(K)$ and $x, y \in K$, we denote the left translate of $f$ by $x$ (resp. the right translate of $f$ by $y$) as $L_xf$ (resp. $R_yf$) and define them as  $L_xf(y) = R_yf(x)=f(x*y):= \int_K f \ d(p_{_x}*p_{_y})$.

	A non-zero Borel measure $\lambda$ on $K$ taking values in $[0, \infty]$ is called a left Haar measure if $(p_{_x}*\lambda)$ is well defined and equals $\lambda$ for each $x\in K$. For each $\mu\in M(K)$ and  $f, g \in \mathcal{B}(K)$, the convolutions are defined  as below for each $x\in K$, whenever they exist.
	\begin{eqnarray*}
		(\mu *f)(x) &=& \int_K f(\widetilde{y}*x) \ d\mu(y). \\
		(f *\mu)(x) &=& \int_K f(x*\widetilde{y}) \ d\mu(y).\\
		(f*g) (x) &=& \int_K f(x*y) g(\widetilde{y}) \ d\lambda(y),
	\end{eqnarray*}
	where $\lambda$ is a fixed left Haar measure on $K$. For each $p\in [1, \infty]$, we denote $L^p(K, \lambda)$ simply as $L^p(K)$. Finally, recall \cite{BH, JE} that the \textit{left-regular representation} of a hypergroup $K$ is the faithful non-degenerate norm-decreasing $*$-homomorphism $T:M(K) \ra B(L^2(K))$ given by $T(\mu)(f) = T_{\mu}(f) := \mu*f$ 
	for each $\mu\in M(K), f\in L^2(K)$. For each $x\in K$, we abuse notation and write $T_{p_{_x}}(f)$ simply as $T_xf$ or $T_x(f)$. Note that for any $x, y\in K, f\in L^2(K)$, we have that $T_x(f)(y) = (p_x*f)(y) = f(\widetilde{x}*y) = L_{\widetilde{x}}f(y)$. 
	Since $T$ is a homomorphism, for each $x, y\in K$ and $f\in L^2(K)$ we also have the following:
	$$T_{(p_x*p_y)} (f) = \int_K T_z(f) \ d(p_{_x}*p_{_y})(z).$$
	
	Next, we recall some definitions and facts from the study of commutative hypergroups and their duals \cite{BH,JE}. A hypergroup $K$ is called \textit{commutative} if for each $x, y \in K$ we have that $(p_x*p_y) = (p_y*p_x)$.

	\begin{defn}
		A bounded measurable function $\gamma: K\rightarrow \mathbb{C}$ is called a \textit{character} on $K$ if $\gamma \nequiv 0$  on $K$ and the following   are satisfied  for each $x, y \in K$.
		\begin{enumerate}
			\item $\gamma(x*y)=\gamma(x)\gamma(y)$.
			\item $\gamma(\widetilde{x}) = \overline{\gamma(x)}$.
		\end{enumerate}
	\end{defn}
	It follows immediately that $||\gamma||_\infty = \gamma(e)=1$ for any continuous character $\gamma$ on $K$. Note that\cite{BH, JE} unlike locally compact abelian groups, $|\gamma(x)|$ need not lie in $\mathbb{T}$ where  $\gamma$ is any non-trivial character  on $K$ and $x\in K\setminus \{e\}$. In fact, one can easily show the following equivalence \cite{BH}:
	$$Z(K) = \{x\in K : |\gamma(x)|=1 \mbox{ for each continuous character } \gamma \mbox{ on } K \}.$$
	
	For a commutative hypergroup $K$, the dual $\widehat{K}$ of $K$ is defined \cite{JE} as the set of all continuous characters on $K$. Recall\cite{BH,JE} that $\widehat{K}$ is given the topology of uniform convergence on compact subsets of $K$, which coincides with the Gelfand topology on the structure space of the commutative Banach $*$-algebra $L^1(K)$. Note that although $\widehat{K}$ is a non-empty locally compact Hausdorff space, unlike locally compact abelian groups, it need not be a hypergroup \cite{BH, JE}.
	
	For any subset $L$ of $K$, the \textit{annihilator} $N(L)$ of $L$  is defined as $N(L):= \{\gamma \in \widehat{K}: \gamma(x)=1 \mbox{ \ for all \ } x \in L\}$. We know that $N(L)$ will   be closed in $\widehat{K}$ for any subset $L\subseteq K$. Moreover, if $L$ is a compact subhypergroup of $K$, then $N(L)$ is open in $\widehat{K}$ \cite{BH}.
	
	Recall \cite{BH} that every commutative hypergroup admits a Haar measure. Let $K$ be a  commutative hypergroup and $\lambda$ be a Haar measure on it. Then the \textit{Fourier transform} on $K$ is a map  $f\mapsto \widehat{f}:L^1(K) \ra C_0(\hk)$  given by
	$$\widehat{f}(\gamma):= \int_K f(x)\overline{\gamma(x)} \ d\lambda(x), $$
	for each $\gamma \in \widehat{K}$. Similarly, for any $\mu \in M(K)$ the \textit{Fourier-Stieltjes transform} on $K$ is a map $\mu\mapsto \widehat{\mu}:M(K) \ra C(\hk)$  given by $\widehat{\mu}(\gamma):= \int_K \overline{\gamma(x)} \ d\mu(x) $ for each $\gamma \in \widehat{K}$. 
	
	Consider the set $\mathcal{S}= \{\gamma \in \hk : |\widehat{\mu}(\gamma)| \leq ||T_{\mu}|| \mbox{ for each } \mu \in M(K) \}$. By the Levitan-Plancherel Theorem for hypergroups \cite{BH,JE} we know that there exists a unique non-negative regular Borel measure $\pi$ on $\hk$ such that for any $f\in L^1(K)\cap L^2(K)$, we have
	$$\int_K |f|^2 \ d\lambda = \int_{\hk} |\widehat{f}|^2 \ d\pi.$$
	In fact, we have that $supp(\pi)=\mathcal{S}$. We call $\pi$ the \textit{Plancherel measure} on $\hk$. Note that unlike locally compact groups, we do not necessarily have that $\mathcal{S}=\hk$. In fact, we have that $\mathcal{S}=\hk$ if   the dual $\hk$ is again a hypergroup, although the converse is not true \cite[Example 9.1C]{JE}.
	
	Throughout this article, whenever we write $\hk$, we consider the measure space $(\hk, \pi)$. Hence for any $p\in [1, \infty]$, whenever we write $L^p(\hk)$, we mean the space $L^p(\mathcal{S}) = L^p(\hk, \pi)$. Finally, recall \cite{JE} that the \textit{inverse Fourier transform} on a commutative hypergroup $K$ is the map $f\mapsto \widecheck{f}: L^1(\hk) \ra C_0(K)$ defined as
	$$\widecheck{f}(x) := \int_{\hk} f(\gamma) \gamma(x) \ d\pi(\gamma).$$

	We conclude this section with a brief account on hypergroup joins\cite{BH, JE}, which provides us with an abundance of useful examples of hypergroups. Given a compact hypergroup $(H, *_H)$ with normalized left Haar measure $\lambda_H$ and a discrete hypergroup $(J, *_J)$ with left Haar measure $\lambda_J$ such that $H\cap J=\{e\}$, where $e$ is the identity of both $H$ and $J$, the hypergroup join $K:=(H\wedge J, *)$ is the space $(H\cup J, *)$, where the convolution is defined as the following:
	\begin{align*}
	p_{_x} * p_{_y} \  = \ p_{_y} * p_{_x} &:= p_{_y}, \mbox{  if  }  x\in H, y \in J\setminus \{e\};\\
	p_{_x} * p_{_y} &:= p_{_x} *_{_H} p_{_y},  \mbox{ if  }  x, y \in H;\\
	p_{_x} * p_{_y} &:= p_{_x} *_{_J} p_{_y},  \mbox{  if  }  x, y \in J, x\neq \widetilde{y};\\
	p_{_x} * p_{_y} &:= c_e \lambda_{_H} + \sum_{_{z\in J\setminus \{e\}}} c_z p_{_z},  \mbox{   if   }  x, y \in J\setminus \{e\}, x= \widetilde{y}, 
	\end{align*}
	where   we have $p_{_x} *_{_J} p_{_{\widetilde{x}}} = \sum_{_{z\in J}} c_z p_{_z}$. The underlying space  $(H\cup J)$ is equipped with the unique (locally compact) topology for which both $H$ and $J$ are closed subspaces of $(H\cup J)$. Observe that $H$ is a compact subhypergroup of $K$. Moreover, the left Haar measure $\lambda$ of $K$ is given by $\lambda= \lambda_{_H} + \chi_{_{J\setminus \{e\}}} \lambda_{_J}$ (see \cite[1.5.13]{BH} for further details).
	
	
	
	\section{Equality in Hausdorff-Young Inequality}
	
	Let $K$ be a commutative hypergroup, $\lambda $ be a Haar measure on $K$ and $\pi$ be the Plancherel measure on $K$ associated with $\lambda$. 
	For any $p\in [1,2]$, we denote by $p'$ the conjugate of $p$, \textit{i.e,} $p'\in \mathbb{R}$ such that $\frac{1}{p} + \frac{1}{p'}=1$ if $1<p\leq 2$ and $p'=\infty$ if $p=1$.

	
	It was shown in \cite{SD} that one can   extend the domain of the Fourier transform to $L^p(K)$ for $1\leq p\leq 2$, as well. In fact, using the density of $C_c(K)$ it can be shown   that for $1\leq p\leq 2$, the Hausdorff-Young transformation $f\mapsto \widehat{f}: L^p(K)\ra L^{p'}(\hk)$ is a well-defined linear mapping such that the Hausdorff-Young inequality holds true \cite[Proposition 2.1]{SD}:
	$$||\widehat{f}||_{p'} \leq ||f||_p.$$
	For $1\leq p\leq 2$ we say that a function $f \in L^p(K)$  \textit{attains equality in Hausdorff-Young}  if we have that $||\widehat{f}||_{p'} = ||f||_p$.
	
	The domain of the inverse Hausdorff-Young transformation $f\mapsto \widecheck{f}$ can also be extended in a similar fashion \cite[Proposition 2.2]{SD} and for $1\leq p \leq 2$, we have that the inverse Hausdorff-Young map $f\mapsto \widecheck{f}: L^p(\hk)\ra L^{p'}(K)$ is a well-defined linear mapping such that the Hausdorff-Young inequality $||\widecheck{f}||_{p'} \leq ||f||_p$ holds true.
	We say that a function $f \in L^p(\hk)$  \textit{attains equality in Hausdorff-Young} on $\hk$  if we have that $||\widecheck{f}||_{p'} = ||f||_p$. The following properties regarding the Hausdorff-Young and inverse Hausdorff-Young maps are proved in \cite{BH,  SD, JE}. 
	
	\begin{rem}\label{remhy}
		Let $K$ be any commutative hypergroup. Then the following assertions hold true for $1\leq p \leq 2$.
		\begin{enumerate}
			\item For $f\in L^1(K)$ such that $\widehat{f}\in L^1(\hk)$,   we have that $\big{(}\widehat{f} \  {\widecheck{\big{)}}} = f$, $\lambda$-almost everywhere on $K$.
			\item For any $f\in L^p(K)$, $\phi \in L^p(\hk)$, we have that
			$ (\widehat{f}\phi)^{\widecheck{}} = f * \widecheck{\phi}$.
			\item For any $f\in L^1(K)$, $\phi \in L^p(\hk)$, we have that
			$ (\widehat{f}\phi)^{\widecheck{}} = f * \widecheck{\phi}$, $\lambda$-almost everywhere on $K$.
		\end{enumerate}
	\end{rem}

	Now recall from the introduction that we say that a hypergroup $K$ abides by the basic uncertainty principle if it is included in the following set:
	\begin{eqnarray*}
		PW_H &:=&  \{ K : K \mbox{ is a commutative hypergroup, and there does not exist }\\
		& & \mbox{any non-trivial function } f \mbox{ on } K \mbox{ such that both } f \mbox{ and } \widehat{f} \mbox{ have }\\
		& & \mbox{ compact support} \}.
	\end{eqnarray*}
	In the first main result that follows, we show that for a commutative hypergroup $K$, if the constant $1$ in Hausdorff-Young inequality is not sharp for some $p\in (1,2)$, then $K\notin PW_H$. We know that the standard proof of Hausdorff-Young inequality is an application of the Riesz-Thorin complex interpolation theorem. Hence in order to study equality in Hausdorff-Young, one requires some careful observations in the techniques used in the proof of Riesz-Thorin's theorem, and in turn investigate the role played by the equality conditions for H\"{o}lder's inequality in the technique.
	\begin{thm} \label{main2}
		Let $1<p<2$. If  $f\in L^p(K)$ attains equality in Hausdorff-Young, then both $f$ and $\widehat{f}$ are compactly supported in $K$ and $\hk$ respectively,  \textit{i.e,} $K\notin PW_H$. Moreover, the supports of $f$ and $\widehat{f}$ are open in $K$ and $\hk$ respectively, as well.
	\end{thm}

	Before we proceed to the proof of the theorem, we need to prove a couple of lemmas and introduce a few objects associated to them. First, adhering to the notions used in the Riesz-Thorin complex interpolation theorem, given any locally compact Hausdorff space $X$ with a positive Borel measure $\mu$ and a function $f\in L^p(X, \mu)$, we define the `\textit{complexification of} $f$' for any fixed $z\in \mathbb{C}$ as the function $S_zf: X \ra \mathbb{C}$ given by:
	$$S_zf(x):= |f(x)|^{p(z+1)/2} \ sgn(f(x)),$$
	for each $x\in X$.
	
	\begin{rem} \label{remtz}
		Define a function $r:\mathbb{C}\setminus \{z\in \mathbb{C}: Re(z)= -1\} \ra \mathbb{R}$ as $r(z):= \frac{2}{Re(z)+1}$. Then for any $f\in L^p(X,\mu)$, $z\in \mathbb{C}$ such that $Re(z) \neq -1$, we have:
		\begin{eqnarray*}
			||S_zf||_{r(z)}^{r(z)} &=& \int_X |S_zf(x)|^{r(z)} \ d\mu(x)\\
			&=& \int_X \Big{|}\ |f(x)|^{\frac{p(Re(z) +1)}{2} + i  p \frac{Im(z)}{2}}\Big{|}^{r(z)} \ d\mu(x)\\
			&=& \int_X |f(x)|^{\frac{p(Re(z) +1)}{2} \ r(z)} \ d\mu(x)\\
			&=& \int_X |f|^p \ d\mu \ = \ ||f||_p^p.
		\end{eqnarray*}
		Hence $S_zf \in L^{r(z)}(X, \mu)$. Furthermore, we have   $|S_{\bar{z}}f|=|S_z\bar{f}|= |f|^{\frac{p(Re(z)+1)}{2}}$.
	\end{rem}
	
	The next lemma is a very useful observation regarding the complexification of simple functions on $K$ and their Fourier transforms. The proof works similarly to the case for abelian locally compact groups\cite{HH}. We include the details of the proof  here for convenience and completeness.
	
	\begin{lem} \label{lementire}
		Let $F$ and $H$ be two simple functions on $K$ and $\hk$ respectively. Fix any $p\in (1, 2)$, and define a complex function $\Omega:\mathbb{C}\ra \mathbb{C}$ as the following:
		$$\Omega(z):= \int_{\hk} {(S_zF)}^{\widehat{}} (\gamma) S_z\widebar{H}(\gamma) \ d\pi(\gamma).$$
		Then $\Omega$ is an entire function.
	\end{lem}
	
	\begin{proof}
		
		Let $F = \sum_{k=1}^N \exp({a_k + i b_k}) \chi_{_{I_k}}$ where $I_1, I_2, \ldots I_N \subseteq K$ are mutually disjoint sets with finite measures  and scalars $a_i, b_j \in \mathbb{R}$, $1\leq i, j\leq N$ for some $N\in \mathbb{N}$. Similarly, let $H = \sum_{k=1}^M \exp({c_k + i d_k}) \chi_{_{J_k}}$, where  $J_1, J_2, \ldots J_M \subseteq \widehat K $ are mutually disjoint sets with finite measure and scalars $c_i, d_j \in \mathbb{R}$, $1\leq i, j\leq M$ for some $M\in \mathbb{N}$. Thus for each $z\in \mathbb{C}$ we have that
		$$S_zF(x) \ = \ |F(x)|^{\frac{p(z+1)}{2}} \ sgn(F(x)) = \sum_{n=1}^N
		\exp \Big{(} \frac{pa_n(z+1)}{2} \Big{)} \exp(ib_n) \chi_{_{I_n}}(x),$$
		for each $x\in K$. Similarly for each $z\in \mathbb{C}$, on $\hk$ we have that
		$$S_z\bar{H}(\gamma) \ = \ |H(\gamma)|^{\frac{p(z+1)}{2}} \ sgn(\bar{H}(\gamma)) =  \sum_{m=1}^M
		\exp \Big{(} \frac{pc_m(z+1)}{2} \Big{)} \exp(-id_m) \chi_{_{J_m}}(\gamma),$$
		for each $\gamma \in \hk$. Hence for each $z\in \mathbb{C}$, one can define $\Omega$ in the following manner:
		\begin{eqnarray*}
			\Omega (z) &=&  \int_{\hk} (S_zF)^{\widehat{}} (\gamma) S_z\bar{H}(\gamma) \ d\pi(\gamma)\\
			&=& \sum_{m=1}^M \int_{J_m} (S_zF)^{\widehat{}}(\gamma) \exp \Big{(} \frac{pc_m(z+1)}{2} \Big{)} \exp(-id_m) \ d\pi(\gamma)\\
			&=& \sum_{m=1}^M  \exp \Big{(} \frac{pc_m(z+1)}{2} - id_m \Big{)} \int_{J_m}\int_K S_zF(x) \overline{\gamma(x)} \ d\lambda(x) \ d\pi(\gamma)\\
			&=& \sum_{m=1}^M  \exp \Big{(} \frac{pc_m(z+1)}{2} - id_m \Big{)}\int_{J_m}\sum_{n=1}^N\int_{I_n}\exp\Big{(}\frac{pa_n(z+1)}{2}\Big{)}\\
			& & \exp(ib_n)\overline{\gamma(x)} \ d\lambda(x) \ d\pi(\gamma)\\
			&=& \sum_{m=1}^M \sum_{n=1}^N  \exp \Big{(}\frac{pc_m(z+1)}{2} + \frac{pa_n(z+1)}{2} + ib_n - id_m\Big{)}\\
			& & \int_{J_m}\int_{I_n} \overline{\gamma(x)} \ d\lambda(x) \ d\pi(\gamma).\\
		\end{eqnarray*}
		\noindent	But, each $I_n$ and $J_m$ is of finite measure and hence
		\begin{eqnarray*}
			\Big{|} \int_{J_m}\int_{I_n} \overline{\gamma(x)} \ d\lambda(x) \ d\pi(\gamma) \Big{|} &\leq &  \int_{J_m}\int_{I_n} |\overline{\gamma(x)}| \ d\lambda(x) \ d\pi(\gamma)\\
			& \leq & \int_{J_m}\int_{I_n} \ d\lambda(x) \ d\pi(\gamma) \ = \ \lambda(I_n)\pi(J_m) \ < \ \infty,
		\end{eqnarray*}
		as $||\gamma||_\infty = 1$ for each $\gamma \in \hk$. Hence it follows immediately that $\Omega$ is entire.
	\end{proof}

	As emphasized while introducing the statement of Theorem \ref{main2} above, the  equality-conditions for the following slightly generalized version of H\"{o}lder's inequality \cite{HH} play a pivotal role in the proof of the theorem. 
	
	\begin{lem} \label{heq}
		Let $1\leq p < \infty$, $f\in L^p(X)$, $g\in L^{p'}(X)$ where $(X, \mu)$ is any measure space where $\mu$ is a non-negative measure. Then equality is attained in a generalized H\"{o}lder's inequality, \textit{i.e,} we have: $$\int_X fg \ d\mu = ||f||_p||g||_{p'} \neq 0 $$ if and only if the following assertions are satisfied:
		\begin{enumerate}
			\item If $p>1$, then $\mu$-almost everywhere on $X$ we have: $$g(x) = ||g||_{p'} ||f||_p^{1-p} |f(x)|^{p-1} sgn (\overline{f(x)}).$$
			\item If $p=1$, then  $\mu$-almost everywhere on $\{x\in X: f(x)\neq 0\}$ we have: $$g(x) = ||g||_\infty \ sgn (\overline{f(x)}).$$
		\end{enumerate}
	\end{lem}
	\begin{proof}
		Note that the given equality implies the following assertions:
		\begin{eqnarray*}
			0 \ \neq \ ||f||_p||g||_{p'} &=& \int_X fg \ d\mu\\
			&=& \Big{|}\int_X fg \ d\mu \Big{|}\\
			&\leq & \int_X |fg| \ d\mu \ \leq \ ||f||_p||g||_{p'}.
		\end{eqnarray*}
		Hence each inequality above must be an equality. In particular, we have that 
		$$\int_X |fg| \ d\mu = \int_X fg \ d\mu = \int_X |fg|sgn(f)sgn(g) \ d\mu.$$
		Hence we must have that $sgn(g(x)) = sgn(\overline{f(x)})$ $\mu$-almost everywhere on $X$. Furthermore, since equality is attained in the classical H\"{o}lder's inequality $\int_X |fg| \ d\mu \leq ||f||_p||g||_{p'}$, we have the required assertion following the equality-conditions in Young's inequality for products.
	\end{proof}

	We know that for any commutative hypergroup $K$, the Fourier transform $\widehat{f}$ lies in $L^{p'}(\hk)$ for any function $f\in L^p(K)$. Now using $\widehat{f}$, we construct a function $f^*$ in $L^p(\hk)$ as the following:
	$$f^*(\gamma ) := |\widehat{f}(\gamma )|^{p'-2} \: \widehat{f}(\gamma ),$$
	for each $\gamma \in \hk$.
	
	\begin{rem} \label{remfstar}
		Note that with this definition, $f^* \in L^p(\hk)$ and $||f^*||_p^p = ||\widehat{f}||_{p'}^{p'}$ since we have:
		\begin{eqnarray*}
			||f^*||_p^p \ = \ \int_{\hk} |f^*(\gamma)|^p \ d\pi(\gamma) &=& \int_{\hk} |\widehat{f}(\gamma )|^{(p'-1)p} \ d\pi(\gamma)\\
			&=& \int_{\hk} |\widehat{f}(\gamma )|^{p'} \ d\pi(\gamma) \ = \ ||\widehat{f}||_{p'}^{p'}.
		\end{eqnarray*}
	\end{rem}
	Finally, the next lemma confirms that for  $1<p<2$ if a function $f\in L^p(K)$ attains equality in Hausdorff-Young on $K$, then the corresponding function $f^* \in L^p(\hk)$ will also attain equality in Hausdorff-Young on $\hk$. Again, the proof of the lemma follows closely the ideas used for a similar result in  locally compact abelian groups \cite{HH}. We include a brief sketch of proof here for the sake of completeness, and highlight the areas where results specific to hypergroups need to be imposed.

	\begin{lem} \label{fstarmax}
		Let $1<p<2$. If $f\in L^p(K)$ attains equality in Hausdorff-Young, then $f^*\in L^{p}(\widehat{K})$ also attains equality in Hausdorff-Young. Moreover,  for each $x\in K$ we have that $$( \widecheck{f^*})(x) = ||f||_p^{p'-p} |f(x)|^{p-2}{f(x)}$$
		for each $x\in K$, $\lambda$-almost everywhere.
	\end{lem}
	
	\begin{proof}
		Recall from Remark \ref{remfstar} that $f^*, \bar{f^*} \in L^p(\hk)$. Since $f\in L^p(K)$ attains equality in Hausdorff-Young, it follows immediately that $||f^*||_p = ||\widehat{f}||_{p'}^{p'/p} = ||f||_p^{p'/p}$. Now we define a linear functional $Q: L^p(K) \ra \mathbb{C}$ as
		$$ Q(\phi) := \int_{\widehat{K}} \widehat{\phi}(\gamma) \overline{ f^*(\gamma)} \ d\pi(\gamma) \ \ \mbox{for each } \phi \in L^p(K).$$
		Note that $Q$ is a bounded linear functional on $L^p(K)$ with $||Q|| = ||f||_p^{p'/p}$ since
		$$Q(f/||f||_p) = \int_{\widehat{K}} ||f||_p^{-1} \widehat{f} |\widehat{f}|^{p'-2} \overline{\widehat{f}} \ d\pi = ||f||_p^{-1} \int_{\widehat{K}} |\widehat{f}|^{p'} \ d\pi = ||f||_p^{p'/p},$$
		and H\"{o}lder's inequality and the Hausdorff-Young inequality for hypergroups \cite{SD} on $L^p(K)$  implies that:
		\begin{eqnarray*}
			||Q|| &=& \sup_{||\phi||_p\leq 1} |	Q(\phi)|\\
			&\leq & \sup_{||\phi||_p\leq 1} \Big{(} \int_{\widehat{K}} |\widehat{\phi}| |f^*| \ d\pi \Big{)}\\
			&\leq & ||f^*||_p \sup_{||\phi||_p\leq 1} \big{(} ||\phi||_p\big{)} = ||f^*||_p = ||f||_p^{p'/p}.
		\end{eqnarray*}
		Now it follows from the duality of $L^p$-spaces that there exists a unique 	$g\in L^{p'}(K)$ such that $||g||_{p'} = ||Q||= ||f||_p^{p'/p}$ and 
		$$ Q(\phi) = \int_K \phi(x)\overline{g(x)} \ d\lambda(x) \ \ \mbox{ for each } \phi \in L^p(K).$$
		On the other hand, since $f^*\in L^p(\hk)$ and $\phi \in L^p(K)$, the generalized Parseval's identity for hypergroups \cite[Proposition 4.1]{SD} implies that:
		$$Q(\phi)= \int_{\widehat{K}} \widehat{\phi}(\gamma) \overline{ f^*(\gamma)} \ d\pi(\gamma) = \int_K \phi(x) \overline{\widecheck{f^*}(x)} \ d\lambda(x).$$
		Thus for each $\phi \in L^p(K)$ we have that
		$$ \int_K \phi(x) \overline{\widecheck{f^*}(x)} \ d\lambda(x) = \int_K \phi(x)\overline{g(x)} \ d\lambda(x) .$$
		Hence $\widecheck{f^*}= g$ $\lambda$-almost everywhere on $K$. Moreover, we have that $f^*$ attains equality in Hausdorff-Young in $L^p(\widehat{K})$ since 
		$$||\widecheck{f^*}||_{p'} = ||g||_{p'} = ||Q|| = ||f||_p^{p'/p}=||f^*||_p.$$
		Next in order to gain an explicit expression for $\widecheck{f^*}$, first we set $\phi=f$ in particular. This implies that:	
		\begin{eqnarray*}
			\int_K f(x) \overline{g(x)} \ d\lambda(x)   =  \ Q(f) &=& \int_{\hk} \widehat{f} (\gamma) \overline{f^*(\gamma)} \ d\pi(\gamma)\\
			&=&\int_{\hk} |\widehat{f}|^{p'} \ d\pi = ||\widehat{f}||_{p'}^{p'} = ||f||_p ||f||_p^{p'-1} = ||f||_p ||\bar{g}||_{p'}.	
		\end{eqnarray*}		
		Thus equality holds in H\"{o}lder's inequality on $K$. Since $p>1$, using Lemma \ref{heq}, for almost all $x\in K$ we have that:
		\begin{eqnarray*}
			\overline{g(x)} &=& ||\bar{g}||_{p'} ||f||_p^{1-p} |f(x)|^{p-1} \ sgn(\bar{f}(x))\\
			&=& ||f||_p^{p'-p} |f(x)|^{p-2} \ \bar{f}(x).
		\end{eqnarray*}		
		Hence  for each $x\in K$, $\lambda$-a.e we have the following structure of $\widecheck{f^*}$:
		$$\widecheck{f^*} (x) = g(x) = ||f||_p^{p'-p} |f(x)|^{p-2} f(x).$$
	\end{proof}
	Now we return to the proof of the first main theorem  (Theorem \ref{main2}) of this article, as discussed before.

	\noi
	{\bf Proof of   \Cref{main2}:}
	
	First, we assume without any loss of generality that $||f||_p =1$.
	
	Let $\mathbb{S}$ denote the closed strip $\{z\in \mathbb{C} : 0\leq Re(z) \leq 1\}$, $\mathbb{S}_0$ denote the interior $\{z\in \mathbb{C} : 0 < Re(z) < 1\}$ of the strip and $\mathbb{S}'$ denote the half-open strip $\{z\in \mathbb{C} : 0 \leq Re(z) < 1\}$. Set $\mathbb{L}:= \mathbb{S}\setminus \mathbb{S}' = \{z\in \mathbb{C} : Re(z)=1\}$.
	
	Recall from Remark \ref{remtz} that we defined $r(z)= \frac{2}{Re(z)+1}$ for each $z\in \mathbb{C}$, $Re(z)\neq -1$. Hence  we have that $1\leq r(z) \leq 2$ for any $z\in \mathbb{S}$. Moreover, we know from the same remark that $S_zf\in L^{r(z)}(K)$ for each $z\in \mathbb{S}$ and $||S_zf||_{r(z)} = ||f||_p^{p/r(z)} = 1$ for each $z\in \mathbb{S}$. Thus the family $\{S_zf\}$ is uniformly bounded on $\mathbb{S}$. Hence using the density properties of simple functions  in  $L^{r(z)}(K)$, we can find\cite{HH} a sequence of simple functions $\{F_n\}$ on $K$ such that the following conditions are satisfied :
	\begin{enumerate}
		\item $||S_zF_n||_{r(z)} \leq ||S_zf||_{r(z)}$ for each $n\in \mathbb{N}$, $z\in \mathbb{S}$.
		\item $\lim_{n\ra \infty} ||S_zF_n - S_zf||_{r(z)} = 0 $ uniformly on compact subsets of $\mathbb{S}$.
	\end{enumerate}
	Similarly, since we know from Remark \ref{remfstar} that $f^*\in L^p(\hk)$, we can find a sequence $\{H_n\}$ of simple functions on $\hk$ such that the following assertions hold true:
	\begin{enumerate}
		\item $||S_zH_n||_{r(z)} \leq ||S_zf^*||_{r(z)}$ for each $n\in \mathbb{N}$, $z\in \mathbb{S}$.
		\item $\lim_{n\ra \infty} ||S_zH_n - S_zf^*||_{r(z)} = 0 $ uniformly on compact subsets of $\mathbb{S}$.
	\end{enumerate}
	Now for each $n\in \mathbb{N}$, we define $\phi_n:\mathbb{C}\ra \mathbb{C}$ as
	$$\phi_n(z) := \int_{\hk} (S_zF_n)^{\widehat{}}(\gamma) S_z\bar{H}_n(\gamma) \ d\pi(\gamma).$$
	It follows immediately from Lemma \ref{lementire} that each $\phi_n$ is an entire function. Now define $\phi: \mathbb{C}\ra \mathbb{C}$  as
	
	$$\phi(z) := \int_{\hk} (S_zf)^{\widehat{}}(\gamma) S_z\bar{f^*}(\gamma) d\pi(\gamma).$$
	Then for any $z\in \mathbb{S}$ and $n\in \mathbb{N}$ we have the following observation:
	\begin{eqnarray*}
		|\phi_n(z) - \phi(z)| & \leq & \int_{\hk} \Big{|} (S_zF_n)^{\widehat{}}(\gamma) S_z\bar{H}_n(\gamma) - (S_zf)^{\widehat{}}(\gamma) S_z\bar{f^*}(\gamma) \Big{|} \ d\pi(\gamma)\\
		& \leq & \int_{\hk} {|} (S_zF_n)^{\widehat{}} - (S_zf)^{\widehat{}}{|}(\gamma) \  {|} S_z\bar{H}_n{|}(\gamma) \ d\pi(\gamma) \\
		& & + \int_{\hk} {|} (S_z\bar{H}_n) - (S_z\bar{f^*}){|}(\gamma)\  {|} (S_zf)^{\widehat{}}{|}(\gamma) \ d\pi(\gamma)\\
		& \leq & ||((S_zF_n - S_zf)^{\widehat{}})||_{r(z)'} \: ||S_z\bar{H}_n||_{r(z)} \\
		& &  + ||(S_z\bar{H}_n) - (S_z\bar{f^*})||_{r(z)} \: ||(S_zf)^{\widehat{}}||_{r(z)'}\\
		& \leq & ||S_zF_n - S_zf||_{r(z)} \: ||S_{\bar{z}}f^*||_{r(\bar{z})}\\
		& &  + ||S_{\bar{z}}H_n-S_{\bar{z}}f^*||_{r(\bar{z})} \: ||S_zf||_{r(z)},
	\end{eqnarray*}
	where the third inequality follows from H\"older's inequality and the fourth inequality follows from the Hausdorff-Young inequality\cite{SD} on $L^p(K)$ for $1\leq p\leq 2$ and the construction of the sequences $\{H_n\}$ and $\{F_n\}$.

	Hence it follows immediately from the choice of the respective sequences $\{H_n\}$ and $\{F_n\}$ on $\hk$ and $K$ that the sequence $\phi_n$ converges to $\phi$ uniformly on compact sets of $\mathbb{S}$. Hence $\phi$ is analytic on
	$\mathbb{S}_0$ and continuous on $\mathbb{S}$.
	
	Using H\"older's inequality for $\phi_n(z)$ for each $n\in \mathbb{N}$, $z\in \mathbb{S}$ we have that
	\begin{eqnarray*}
		|\phi_n(z)| & \leq & ||(S_zF_n)^{\widehat{}}||_{r(z)'} \ ||S_{\bar{z}}H_n||_{r(z)}\\
		&\leq & ||S_zF_n||_{r(z)} \ ||S_{\bar{z}}H_n||_{r(z)}\\
		&\leq & ||S_z f||_{r(z)} \ ||S_{\bar{z}}f^*||_{r(\bar{z})}\\
		&=& ||f||_p^{p/r(z)} \ ||f^*||_p^{p/r(z)}\\
		&=& ||f||_p^{p/r(z)} \ ||\widehat{f}||_{p'}^{p'/r(z)}\\
		&\leq & ||f||_p^{p/r(z)} \ ||{f}||_{p}^{p'/r(z)} \ = \ ||f||_p^{\frac{p+p'}{r(z)}} \ = \ 1 ,
	\end{eqnarray*}
	where the second and fourth inequality follows from the Hausdorff-Young inequality on $L^{r(z)}(K)$ and $L^p(K)$ respectively, since $1\leq r(z) \leq 2$, the third inequality follows from the choice of the sequences $\{F_n\}$ and $\{H_n\}$ , the first equality follows from Remark \ref{remtz} since $r(\bar{z})=r(z)$ and finally, the second equality follows from Remark \ref{remfstar}. Hence we have that $|\phi(z)| \leq 1$ for each $z\in \mathbb{S}$.
	
	In particular, note that $(-1+ 2/p) \in \mathbb{S}_0$ since $1<p<2$. Also, $S_{(-1+2/p)}f = |f| \ sgn(f) =f$ and $S_{(-1+2/p)}\bar{f^*} = |\bar{f^*}| \ sgn(\bar{f^*}) = \bar{f^*}$. Hence at $z = (-1+ 2/p) \in \mathbb{S}_0$, we have that:
	
	\begin{eqnarray*}
		\phi(z) & = & \int_{\hk} \widehat{f} (\gamma) \ \bar{f^*} (\gamma) \ d\pi(\gamma)\\
		& = & \int_{\hk} \widehat{f} (\gamma) \: \big{|}{\widehat{f}(\gamma)}\big{|}^{p'-2}   \ \overline{\widehat{f}(\gamma)} \ d\pi(\gamma)\\
		& = & \int_{\hk} |\widehat{f}(\gamma)|^{p'} \ d\pi(\gamma)\\
		& = & ||\widehat{f}||_{p'}^{p'} \ = \ ||f||_p^{p'} \ = \ 1,
	\end{eqnarray*}
	where the second-last equality is attained since $f$ attains equality in Hausdorff-Young. 
	
	Hence we must have that $\phi \equiv 1$ on $\mathbb{S}$. Proceeding in the same manner as $\phi_n(z)$, for each $z\in \mathbb{S}$ we have the following:
	\begin{eqnarray*}
		1 \ = \ \phi(z) \ = \ \int_{\hk} (S_zf)^{\widehat{}}(\gamma) S_z\bar{f^*}(\gamma) d\pi(\gamma) &\leq & ||S_{\bar{z}}f^*||_{r(\bar{z})} \ ||(S_zf)^{\widehat{}}||_{r(z)'}\\
		&\leq & ||S_{\bar{z}}f^*||_{r(\bar{z})} \ ||(S_zf)||_{r(z)}\\
		& = & ||f^*||_p^{p/r(\bar{z})} \ ||f||_p^{p/r(z)}\\
		& = & ||f||_p^{p'/r(z)} \ ||f||_p^{p/r(z)} \ = \ 1,
	\end{eqnarray*}
	where the second-last equality  follows from Remark \ref{remfstar} since $f$ attains equality in Hausdorff-Young in $L^p(K)$. Hence for each $z\in \mathbb{S}$ we must have that
	$$ \ \int_{\hk} S_z\bar{f^*}(\gamma) (S_zf)^{\widehat{}}(\gamma)  \ d\pi(\gamma) = ||S_{{z}}\bar{f^*}||_{r({z})} \ ||(S_zf)^{\widehat{}}||_{r(z)'} = 1.$$
	Note that $r(z) >1$ for each $z\in \mathbb{S}'$ and $r(z)=1$ for $z\in \mathbb{L}$. Hence it follows from Lemma \ref{heq} that for each $z\in \mathbb{S}'$, $(S_zf)^{\widehat{}}$ must have the following structure:
	$$(S_zf)^{\widehat{}}(\gamma) = ||(S_zf)^{\widehat{}}||_{r(z)'} \ ||S_z\bar{f^*}||_{r(z)}^{1-r(z)} |S_z\bar{f^*}(\gamma)|^{r(z)-1} \ sgn\Big{(}\overline{S_z\bar{f^*}(\gamma)}\Big{)},$$
	for all $\gamma \in \hk$ $\pi$-almost everywhere. But it follows from Remarks \ref{remtz} and \ref{remfstar} that for each $z\in \mathbb{S}$ we have 
	$$||S_z\bar{f^*}||_{r(z)} = ||\bar{f^*}||_p^{p/r(z)} = ||f||_p^{p'/r(z)} = 1.$$
	Hence $||(S_zf)^{\widehat{}}||_{r(z)'} = 1$ as well. Also recall that it follows immediately from the construction of $S_zg$ for each $g\in L^s(K)$, $1\leq s \leq 2$, that 
	$$|S_zg| = |S_z\bar{g}|= |S_{\bar{z}}g| = |g|^{\frac{s(Re(z)+1)}{2}},$$
	and similarly, $sgn(S_{\bar{z}}g) = |g|^{-i  \frac{s  Im(z)}{2}} \ sgn(g) = \big{(}sgn(S_z\bar{g})\big{)}^{-1} = sgn\big{(}\overline{S_z\bar{g}}\big{)}$. Thus for any $z\in \mathbb{S}'$,  we have that
	\begin{equation} \label{E1}
	(S_zf)^{\widehat{}}(\gamma) = |S_{\bar{z}}{f^*}(\gamma)|^{r(z)-1} \ sgn\big{(}{S_{\bar{z}}{f^*}(\gamma)}\big{)}
	\end{equation}
	for all $\gamma \in \hk$, $\pi$-almost everywhere. Similarly, since $r(z)=1, r(z)' = \infty$ for any $z\in \mathbb{L}$, we have that $||(S_zf)^{\widehat{}}||_{\infty} = 1$. Hence using Lemma \ref{heq}, for any $z\in \mathbb{L}$, we have that
	\begin{equation} \label{E2}
	(S_zf)^{\widehat{}}(\gamma) = sgn \Big{(}{S_{\bar{z}}{f^*}(\gamma)}\Big{)}
	\end{equation}
	for all $\gamma\in \{\theta \in \hk : S_z\bar{f^*}(\theta)\neq 0\}$ , $\pi$-almost everywhere.
	
	Now for each $z\in \mathbb{L}$, consider the set $E_z:= \{\gamma \in \mathcal{S} : |(S_zf)^{\widehat{}}(\gamma)|=1\} \subset \hk $. It follows from Equation \ref{E2} that $S_{\bar{z}}{f^*}(\gamma) = 0$ for every $\gamma\in E_z^c$, $\pi$-almost everywhere.
	
	On the other hand, we know from Lemma \ref{fstarmax} that $f^*\in L^p(\hk)$ attains equality in Hausdorff-Young, \textit{i.e,} $||f^*||_p = ||\widecheck{f^*}||_{p'}$. Moreover, we have that $||f^*||_p=1$ since $||f||_p=1$ and for each $x\in K$ we have:
	\begin{eqnarray*}
		{|} \widecheck{f}^* (x){|}^{p'-1} \ sgn(\widecheck{f}^*(x)) &=& |f(x)|^{(p-1)(p'-1)} \ sgn(f(x))\\
		&=& |f(x)| \ sgn(f(x)) \ = \ f(x).
	\end{eqnarray*}
	Hence we can proceed in similar fashion, interchanging $f$ and $f^*$, and thereby  defining $\phi(z) = \int_K (S_zf^*)^{\widecheck{}} (x) S_z\bar{f}(x) \ d\lambda(x)$, to eventually use Lemma \ref{heq} in order to conclude that for each $z\in \mathbb{S}'$ we have that:
	\begin{equation} \label{E1'}
	(S_zf^*)^{\widecheck{}}(x) = |S_{\bar{z}}{f}(x)|^{r(z)-1} \ sgn\Big{(}{S_{\bar{z}}{f}(x)}\Big{)}
	\end{equation}
	for all $x\in K$, $\lambda$-almost everywhere. Similarly, for any $z\in \mathbb{L}$, we have that:
	\begin{equation} \label{E2'}
	(S_zf^*)^{\widecheck{}}(x) = sgn \Big{(}{S_{\bar{z}}{f}(x)}\Big{)}
	\end{equation}
	for all $x\in \{y\in K : S_z\bar{f}(y)\neq 0\}$ , $\lambda$-almost everywhere.
	
	Now for each $z\in \mathbb{L}$, consider the set $D_z:= \{x\in K : |(S_zf^*)^{\widecheck{}}(x)|=1\} \subset K $. It follows from Equation \ref{E2'} that $S_{\bar{z}}{f}(x) = S_zf(x) = 0$ for every $x\in D_z^c$, $\lambda$-almost everywhere.
	
	In particular, consider $z=1 \in \mathbb{L}$. Note that $S_1f = |f|^p \ sgn(f)$ and $S_1f^* = |f^*|^p \ sgn(f^*)$. Hence for $\lambda$-almost everywhere on $ K$, we have that if $x\notin D_1$, then $|f(x)|^p  = |S_1f(x)| =0$, and hence $f(x)=0$. Similarly, for $\pi$-almost everywhere on $\hk$, we have that if $\gamma\notin E_1$, then 
	$$ |\widehat{f}(\gamma)|^{p'} = |\widehat{f}(\gamma)|^{(p'-1)p} = |f^*(\gamma)|^p  = |S_1f^*(\gamma)| =0,$$
	and hence $\widehat{f}(\gamma)=0$. Thus we see that the essential supports of $f$ and $\widehat{f}$ are included in $D_1$ and $E_1$ respectively.
	
	Note that since $S_1f^*\in L^1(\hk)$ by Remarks \ref{remtz} and \ref{remfstar}, we have that $(S_1f^*)^{\widecheck{}} \in C_0(K)$ and therefore Hausdorff-Young inequality on $L^1(\hk)$  implies that
	$$|(S_1f^*)^{\widecheck{}}(x)| \leq ||(S_1f^*)^{\widecheck{}}||_\infty \leq ||S_1f^*||_1 = ||f^*||_p^p=1,$$
	for each $x\in K$. Hence we can rewrite $D_1$ as $D_1 = \{x\in K : |(S_1f^*)^{\widecheck{}}(x)|\geq 1\}$. Since $(S_1f^*)^{\widecheck{}} \in C_0(K)$, we have that $D_1$ is compact, \textit{i.e,} $f$ is compactly supported.
	
	In a similar manner, since $S_1f \in L^1(K)$, we have that $|(S_1f)^{\widehat{}}(\gamma)| \leq ||S_1f||_1 = 1$ for each $\gamma\in \hk$. Hence we can rewrite $E_1$ as $ E_1= \{\gamma \in \mathcal{S} : |(S_1f)^{\widehat{}}(\gamma)|\geq 1\}\subseteq \{\gamma \in \hk : |(S_1f)^{\widehat{}}(\gamma)|\geq 1\}$. Since $(S_1f)^{\widehat{}} \in C_0(\hk)$, we conclude that  $E_1$ is compact as well, \textit{i.e,} $\widehat{f}$ is compactly supported in $\hk$.
	
	Now to prove the second part of the theorem, first note that since both $f$ and $\widehat{f}$ are compactly supported, without loss of generality we may assume that both $f$ and $\widehat{f}$ are continuous functions, and therefore so are $f^*, S_zf, S_zf^*$ for any $z\in \mathbb{S}$. 
	
	Moreover, we now have that the equalities shown in Equation \ref{E1} and Equation \ref{E1'} hold for every $z\in \mathbb{S}'$, and for each $\gamma \in \mathcal{S}$ and $x\in K$ respectively.
	
	Next, we fix some $\gamma \in \mathcal{S}$ and $v\in \mathbb{R}$. Let $\{z_n\}$ be a sequence in $\mathbb{C}$ such that $Re(z_n) < 1$ for every $n$ and $z_n \ra (1+iv)$ as $n\ra \infty$. Hence it follows from Equation \ref{E1} that for each $\gamma \in \mathcal{S}$ we have that
	\begin{eqnarray*}
		(S_{(1+iv)}f)^{\widehat{}} (\gamma) &=& \lim_{n\ra\infty} (S_{z_n}f)^{\widehat{}} (\gamma) \\
		&=& \lim_{n\ra\infty} |S_{\bar{z_n}}f^*(\gamma)|^{r(z_n)-1} \ sgn \Big{(} S_{\bar{z_n}}f^*(\gamma)\Big{)}\\
		&=& |S_{(1-iv)}f^*(\gamma)|^{r(1+iv)-1} \ sgn \Big{(} S_{(1-iv)}f^*(\gamma)\Big{)}\\
		&=& sgn \Big{(} S_{(1-iv)}f^*(\gamma)\Big{)}.
	\end{eqnarray*} 
	Similarly, for each $x\in K$ using Equation \ref{E1'} we have that
	$$ (S_{(1+iv)}f^*)^{\widecheck{}} (x) =   sgn \Big{(} S_{(1-iv)}f(x)\Big{)}.$$
	
	Now setting $v=0$, we see that if $\gamma \in E_1$, then $$ sgn(\widehat{f}(\gamma)) = sgn(f^*(\gamma)) = sgn(S_1f^*(\gamma)) = (S_1f)^{\widehat{}}(\gamma)\neq 0,$$ and hence $\widehat{f}\neq 0$. Similarly, if $x \in D_1$, then we have that
	$$ sgn(f(x)) = sgn(S_1f(x)) = (S_1f^*)^{\widecheck{}}(x) \neq 0 ,$$
	and hence $f(x)\neq 0$. 
	
	Thus we have that $supp(f) = D_1$ and $supp(\widehat{f})=E_1$. In particular, since $D_1^c = f^{-1}\{0\}$ is closed, we have that $D_1$  is open as well. Similarly, we have that $E_1$ is open in $\hk$. 
	
	Hence the supports $D_1$ and $E_1$ of $f$ and $\widehat{f}$  are compact open subsets of $K$ and $\widehat{K}$ respectively. \qed
	
	\begin{rem}
		In view of the above result, we see that the constant $1$ can never be the best constant in the Hausdorff-Young inequality for $L^p(K)$, $1<p<2$, for a non-compact connected commutative hypergroup $K$ (in particular, for one dimensional hypergroups $(\mathbb{R}^+, *)$ and  other connected non-compact double-coset spaces $G//H$ where $H$ is a compact subgroup of a locally compact group $G$).
	\end{rem}
	
	Next, in order to acquire the other direction of our desired characterization, given any $K\notin PW_H $, we try to find some non-trivial function $f\in L^p(K)$ which will attain equality in Hausdorff-Young. For this reason, it would be useful to gain some insight into the structure and construction of such non-trivial functions. 
	
	As corollaries to the above theorem, roughly speaking, we see that such a function looks ``very similar" to a characteristic function supported on a compact open subhypergroup, or a translate of such a subhypergroup. The proof of the results are inspired by the related results in the category of locally compact abelian groups \cite{HH, HR2}. However, note that unlike the classical group-case, all the characters of a hypergroup are not unitary in nature, \textit{i.e, } for each $\gamma \in \hk$, $x\in K\setminus\{e\}$, the value $\gamma(x)$ need not necessarily sit in the unit circle $\mathbb{T}$, and in fact, may equal $0$. Moreover, we know that the left-regular representations do not provide isometries in the case of hypergroups. Hence adjustments need to be made in order to include the functions $f\in C_c(X)$ for which the identity $e$ of $K$ is not included in the support of $f$. 
	
	We first recall an elementary result in measure theory that we will use frequently in the proof of the immediate corollary. The result follows immediately from the fact that each point on the unit circle is extremal in nature. 
	
	\begin{lem}\label{lemm}
		Let $X$ be a locally compact Hausdorff space and $\mu$ be a regular Borel probability measure on $X$. For any continuous function $\phi:X\ra \mathbb{D}$, if there exists some $\alpha\in \mathbb{T}$ such that $\int_X \phi \ d\mu = \alpha$, then we must have that $\phi(x) = \alpha$ for each $x\in supp(\mu)$.
	\end{lem}
	
	\begin{cor}\label{cor0}
		Let $1<p<2$. If $f\in L^p(K)$ attains equality in Hausdorff-Young, then the following statements hold true.
		\begin{enumerate}
			\item If $x\in supp(f)$, $\gamma\in supp(\widehat{f})$, then $\gamma(x)\in \mathbb{T}$.
			\item There exists a compact open subhypergroup $H$ in $K$ and some $x_0\in K$ such that $supp(f) = \{x_0\}*H$.
			\item  If $e\in supp(f)$, then  $supp(f)$ is a compact open subhypergroup of $K$.
		\end{enumerate}
	\end{cor}
	\begin{proof}
		Unless mentioned otherwise, we assume the same notations here as in the proof of Theorem \ref{main2}.  In particular, we have:
		$$supp(\widehat{f}) = E_1 = \{\gamma\in \mathcal{S} : |(S_1f)^{\widehat{}}(\gamma)|=1\}.$$
		Now for each $\gamma\in E_1$, set $a_\gamma := (S_1f)^{\widehat{}}(\gamma)= sgn\big((S_1f)^{\widehat{}}(\gamma)\big)\in \mathbb{T}$. Thus for each $\gamma\in supp(\widehat{f})$ we have:
		\begin{eqnarray*}
			\int_{K} (S_1f)(x) a_{\gamma}^{-1} \ \overline{\gamma}(x) \ d\lambda(x) &=&  a_\gamma^{-1} \int_{K} (S_1f)(x) \bar{\gamma}(x) \ d\lambda(x)\\
			&=& a_\gamma^{-1} (S_1f)^{\widehat{}}(\gamma)\\
			&=& |(S_1f)^{\widehat{}}(\gamma)|\\
			&=&  1 \ = \ ||f||_p \ |a_\gamma^{-1}| \ ||\bar\gamma||_\infty \ = \ ||S_1f||_{_1} \  ||a_\gamma^{-1}\bar\gamma||_\infty. 
		\end{eqnarray*}
		Hence it follows from Lemma \ref{heq} that $a_\gamma^{-1}\bar\gamma(x) = sgn(\overline{ S_1f(x)})= sgn(\overline{ f(x)})$, $\lambda$-a.e on $supp(S_1f)=supp(f)$, i.e, we have that $\gamma(x) = a_\gamma^{-1} sgn\big(S_1f(x)\big)$,  for each $x\in supp(f)$, $\lambda$-a.e. In particular, for each $\gamma\in supp(\widehat{f})$,  we must have that $|\gamma(x)|=|sgn\big(S_1f(x)\big)|$, $\lambda$-a.e on $supp(f)$. But $\gamma$ is continuous, and hence $|\gamma(x)|=1$ for each $x\in supp(f)$, as required for the assertion of  $(1)$.

		Recall from Theorem \ref{main2} that without loss of generality, we may assume that both $f$ and $\hat{f}$ are continuous. Moreover,  $f(x)\neq 0$ for each $x\in supp(f)$ and $\widehat{f}(\gamma)\neq 0$ for each $\gamma\in supp(\widehat{f})$. Since $supp(\widehat{f})\subseteq \mathcal{S}$ is non-empty and open, it follows immediately from $(1)$ that for each $x\in supp(f)$, we must have that $||\widehat{p}_{_{\widetilde{x}}}||_{_\infty}=1$. Again, recall that $supp(f) = D_1 = \{x\in K : |(S_1f^*)^{\widecheck{}}(x)|=1\}$.

		For each $x\in D_1$, set $\alpha_x:= (S_1f^*)^{\widecheck{}}(x)= sgn\big((S_1f^*)^{\widecheck{}}(x)\big)\in \mathbb{T}$. Then proceeding similarly as above, we have that
		$$\int_{\hk} (S_1f^*) \alpha_x^{-1}\widehat{p}_{_{\widetilde{x}}}  \ d\pi = \alpha_x^{-1} \int_{\hk} (S_1f^*)(\gamma) \gamma(x) \ d\pi(\gamma) =  ||(S_1f^*)||_{_1} \ ||\alpha_x^{-1}\widehat{p}_{_{\widetilde{x}}}||_{_\infty}.$$
		Hence using Lemma \ref{heq} we have that $\alpha_x^{-1} \gamma(x) = sgn(\overline{ S_1f^*(\gamma)})= sgn(\overline{\widehat{f}(\gamma)})$, $\pi$-a.e on $supp(\widehat{f})$. Furthermore, since both $\widehat{f}$ and the map $(x, \gamma)\mapsto \gamma(x) :K\times \hk \ra \mathbb{C}$ is continuous\cite{BH}, we have that 
		$$\gamma(x) sgn(\widehat{f}(\gamma)) = \alpha_x,$$
		for each $\gamma \in supp(\widehat{f})$, where $\alpha_x\in \mathbb{T}$. On the other hand, if we have that $\gamma(x_0) sgn(\widehat{f}(\gamma))= \alpha$ a.e on $supp(\widehat{f})$ for some $\alpha\in \mathbb{T}$, $x_0\in K$, then we immediately have that
		\begin{eqnarray*}
			(S_1f^*)^{\widecheck{}}(x_0) &=& \int_{\hk} S_1f^*(\gamma) \gamma(x_0) \ d\pi(\gamma) \\
			&=& \alpha \int_{E_1} S_1f^*(\gamma) sgn ({\widehat{f}(\gamma)})^{-1} \ d\pi(\gamma)\\
			&=& \alpha \int_{E_1} S_1f^*(\gamma) sgn (S_1f^*(\gamma))^{-1} \ d\pi(\gamma)\\
			&=& \alpha \int_{E_1} |S_1f^*(\gamma)|  \ d\pi(\gamma) \ = \ \alpha \int_{\hk} |\hf|^{p'} d \pi  \ = \ \alpha.
		\end{eqnarray*}
		Hence $x_0\in D_1$, i.e, we have that $x\in supp(f)$ if and only if there exists some scalar $\alpha_x\in \mathbb{T}$ such that $\gamma(x) sgn(\widehat{f}(\gamma)) = \alpha_x$, $\pi$-a.e on $supp(\hf)$. 
		
		Now pick any $x_0\in supp(f)$ and set $H:=\{\widetilde{x}_0\}* supp(f)$. For any $z\in H$, there exists some $y_z\in supp(f)$ such that $z\in \{\widetilde{x}_0\}*\{y_z\}$. Since $\int_K \gamma \ d(p_{_{\widetilde{x}_0}}*p_{_{y_z}}) = \overline{\gamma(x_0)} \gamma(y_z)\in \mathbb{T}$ for each $\gamma\in supp(\hf)$ by $(1)$, using Lemma \ref{lemm} we see that $\gamma(z)= \overline{\gamma(x_0)} \gamma(y_z)$ for each $\gamma\in supp(\hf)$. Hence in particular, $\gamma(z)\in \mathbb{T}$ for each $z\in H, \gamma\in supp(\hf)$.
		
		Similarly, for each $s\in \{x_0\}*H$, there exists some $z_s\in H$ such that $s\in \{x_0\}*\{z_s\}$, and hence by Lemma \ref{lemm} we must have that $\gamma(s)=\gamma(x_0)\gamma(z_s)=\gamma(y_{z_s})\in \mathbb{T}$ for each $\gamma \in supp(\hf)$ for some $y_{_{z_s}}\in supp(f)$. Hence in particular,   we have that
		$$\gamma(s) sgn(\hf(\gamma)) = \gamma(y_{z_s})sgn(\hf(\gamma)) = \alpha_{y_{_{z_s}}} \in \mathbb{T},$$
		for each $\gamma\in supp(\hf)$, i.e, $s\in supp(f)$. Hence we have that $supp(f)= \{x_0\}*H$ as well, since \begin{eqnarray*}
			supp(f)  \supseteq \ \{x_0\}*H  &=& \{x_0\}*\big(\{\widetilde{x}_0\}*supp(f)\big) \\
			&=& \big(\{x_0\}*\{\widetilde{x}_0\}\big)*supp(f)  \supseteq  \{e\}*supp(f)=supp(f).
		\end{eqnarray*}
		Note that $H$ is open since $supp(f)$ is open in $K$ \cite[Lemma 4.1D]{JE}. Furthermore, since $H$ is compact by construction \cite[Lemma 3.2B]{JE}, in order to show that $H$ is a subhypergroup, it is sufficient to show that $H*H\subseteq H$ \cite{BH}. Pick any $z_1, z_2\in H$. Then $\gamma(s_i)=\gamma(x_0)\gamma(z_i)$ for each $s_i\in \{x_0\}*\{z_i\}$ on $supp(\hf)$, $i=1, 2$. Since   $\{x_0\}*\{z_i\}\subseteq supp(f)$ for $i=1, 2$ as well, we have that there exists scalars $\alpha_1, \alpha_2\in \mathbb{T}$ such that $$\alpha_i = \gamma(s_i) sgn(\widehat{f}(\gamma)) = \gamma(x_0)\gamma(z_i) sgn(\widehat{f}(\gamma)),$$
		for each $\gamma \in supp(\hf)$, $i=1, 2$. Now pick any $z\in \{z_1\}*\{z_2\}$.  By Lemma \ref{lemm} we have that $\gamma(z)= \gamma(z_1)\gamma(z_2)\in \mathbb{T}$ for each $\gamma\in supp(\hf)$. Hence in particular, for each $s\in \{x_0\}*\{z\}$ we have that
		\begin{eqnarray*}
			\gamma(s)	sgn(\widehat{f}(\gamma))&=& \gamma(x_0)\gamma(z) sgn(\widehat{f}(\gamma))\\
			&=& \gamma(x_0)\gamma(z_1)\gamma(z_2) sgn(\widehat{f}(\gamma))\\
			&=& \alpha_1 \alpha_2 \alpha_{x_{_0}}^{-1} \in \mathbb{T},
		\end{eqnarray*}
		for each $\gamma\in supp(\hf)$. Hence $\{x_0\}*\{z\}\subseteq supp(f)$, i.e, $z\in \{\widetilde{x}_0\}*supp(f)=H$ \cite[Lemma 4.1B]{JE}. Since $z\in \{z_1\}*\{z_2\}$ was chosen arbitrarily, we have that $\{z_1\}*\{z_2\}\subseteq H$, i.e, $H*H\subseteq H$, which completes the proof of $(2)$.
		
		Part $(3)$ is immediate by choosing $x_0=e$.
	\end{proof}
	
	\begin{rem}
		If the hypergroup under consideration is strong, i.e, if $\hk$ is a hypergroup as well (see \cite{BH} for details on strong hypergroups), then it can be shown in similar fashion that there exists some $\gamma_0\in \hk$ and a compact open subhypergroup $L\subseteq \hk$ such that $supp(\hf)= \{\gamma_0\}* L$, whenever $f\in L^p(K)$ attains 	equality in Hausdorff-Young for some $1<p<2$.
		
	\end{rem}

	\begin{cor} \label{cor1}
		Let $1<p<2$. If $f\in L^p(K)$ attains equality in Hausdorff-Young, then there exists  some scalar $\alpha\in \mathbb{C}$, $x_0\in K$, $\gamma_0\in \hk$ and a compact open subhypergroup $H \subseteq K$ such that $$f(x) = \alpha \  \gamma_0(x) \ \chi_{_{\{x_0\}*H}}(x),$$
		for each $x\in K$, $\lambda$-a.e, where $\gamma_0(x)\in \mathbb{T}$ for each $x\in \{x_0\}*A$.
	\end{cor}

	\begin{proof}
		Throughout this proof, we use the same notations and objects as in the proof of Theorem \ref{main2}. Recall that the supports of both $f$ and $\widehat{f}$ are compact open sets in $K$ and $\hk$ respectively, and we may assume without loss of generality that both $f$ and $\widehat{f}$ are continuous functions.

		Next, recall that for any $v\in \mathbb{R}$ we have that $|(S_{(1+iv)}f)^{\widehat{}} (\gamma)|=1$ whenever $S_{(1-iv)}f^*(\gamma) \neq 0$. In particular, pick any $\gamma_0\in E_1= \mbox{supp}(\widehat{f})$. Then $S_{(1-iv)}f^*(\gamma_0) \neq 0$. Hence we have the following set of inequalities:
		\begin{eqnarray*}
			1 \ = \ |(S_{(1+iv)}f)^{\widehat{}}(\gamma_0)| &=& \Big{|} \int_K S_{(1+iv)}f(x) \bar{\gamma}_0(x) \ d\lambda(x) \Big{|}\\
			&\leq & \int_{D_1} \Big{|}  S_{(1+iv)}f(x) \bar{\gamma}_0(x)\Big{|} \ d\lambda(x)\\
			&\leq & ||S_{(1+iv)}f \bar{\gamma}_0||_1 \ ||\chi_{_{D_1}}||_\infty\\
			&=& \int_K |S_{(1+iv)}f|(x) |\bar{\gamma}_0|(x) \ d\lambda(x)\\
			&\leq & ||S_{(1+iv)}f||_1 \ ||\bar{\gamma}_0||_\infty = ||f||_p^p \ = \ 1,
		\end{eqnarray*}
		where the second and third inequality follows from H\"older's inequality. Hence in particular, we must have that 
		$$ \Big{|} \int_{D_1} S_{(1+iv)}f(x) \bar{\gamma}_0(x)\ d\lambda(x)\Big{|}  = ||S_{(1+iv)}f \bar{\gamma}_0||_1 \ ||\chi_{_{D_1}}||_\infty = 1. $$
		Using the criterion described in Lemma \ref{heq}, since $f$ is continuous on $D_1$ and $f(x) \neq 0$ for each $x\in D_1$, we have that there exists some scalar $c_{(\gamma_0, v)} \in \mathbb{T}$ such that
		$$ sgn \Big{(}S_{(1+iv)}f(x) \bar{\gamma}_0(x)\Big{)} = c_{(\gamma_0, v)}.$$
		Note that the scalar is independent of the choice of $x$ in $D_1$. 
		
		In particular, taking $v=0$ we get that 
		\begin{equation} \label{E5}
		c_{(\gamma_0, 0)} = sgn  \Big{(}S_{1}f(x) \bar{\gamma}_0(x)\Big{)} = sgn(f(x)) \ sgn(\bar{\gamma}_0(x)),
		\end{equation}
		for each $x\in D_1$. Now, pick some $v_0 \in \mathbb{R}\setminus \{0\}$. Then for each $x\in D_1$ we have that $sgn(S_{(1+iv_0)}f) = |f|^{i\frac{pv_0}{2}} sgn(f)$ and hence
		\begin{eqnarray*}
			c_{(\gamma_0, v_0)} &=& sgn \Big{(}S_{(1+iv_0)}f(x) \bar{\gamma}_0(x)\Big{)}\\
			&=& |f(x)|^{i\frac{pv_0}{2}} sgn(f(x)) sgn(\bar{\gamma}_0(x))\\
			&=& |f(x)|^{i\frac{pv_0}{2}} c_{(\gamma_0, 0)},
		\end{eqnarray*}
		where the last equality follows from Equation \ref{E5}. Hence for any $x\in D_1$, we have that $$|f(x)|^{i\frac{pv_0}{2}} = c_{(\gamma_0, v_0)} \ c_{(\gamma_0, 0)}^{-1}.$$
		Setting $\alpha_0 := \Big{(}\frac{c_{(\gamma_0, 0)}}{c_{(\gamma_0, v_0)}}\Big{)}^{\frac{2i}{pv_0}}$ we have that $|f| \equiv \alpha_0$ on $D_1$.  Since $sgn (\bar{\gamma}_0(x)) = \Big{(} sgn(\gamma_0(x))\Big{)}^{-1}$ for any $x\in D_1$, we have the following structure of $f$ on $K$:
		\begin{eqnarray*}
			f(x) &=& |f(x)| \ sgn(f(x))\\
			&=& \alpha_0 c_{(\gamma_0, 0)} \ sgn(\gamma_0(x)) \ = \ \alpha  sgn(\gamma_0(x)),
		\end{eqnarray*}
		setting $\alpha := \alpha_0 c_{(\gamma_0, 0)} $, where the second equality follows from Equation \ref{E5}. Finally, since $f(x) = 0$ for any $x\in D_1^c$, the conclusion follows from Corollary \ref{cor0}, setting $H = \{\widetilde{x}_0\}*D_1$ for some $x_0\in D_1$.  
	\end{proof} 
	
	\begin{rem}
		In the classical theory of locally compact abelian groups, we know\cite{HH} that  for $1<p<2$, a function $f\in L^p(G)$ attains equality in Hausdorff-Young if and only if   $f(x) = \alpha  \gamma_0(x)  \chi_{_{x_0B}}(x)$ on $ G$, for  some scalar $\alpha\in \mathbb{C}$, $\gamma_0\in \widehat{G}$, $x_0\in G$ and a compact open subgroup $B \subseteq G$. 
		
		In the setting of hypergroups, we immediately observe that every function of such form $f = \alpha  \gamma_0  \chi_{_{\{x_0\}*B}}$ will not attain equality in Hausdorff-Young, where $\alpha\in \mathbb{C}, \gamma_0\in \hk, x_0\in K$ are chosen arbitrarily and $B$ is any compact open subhypergroup of $K$. This is true since there exists $\gamma_0\in \hk$ such that  $|\gamma_0|\nequiv 1$ on $\{x_0\}*B$, unless $K$ is a locally compact abelian group (consider the simple examples \cite[9.1B, 9.1C, 9.1D]{JE}).

		
	\end{rem}
	
	Now keeping the findings of Corollary \ref{cor1} in mind, an obvious candidate for a function that attains equality in Hausdorff-Young, would be a characteristic function supported on an appropriate subset of $K$. In particular, we have the following observation.

	\begin{thm} \label{main1}
		Let $A$ be a compact open subhypergroup of $K$. Then $\chi_{_A}$ attains equality in Hausdorff-Young for each $p\in [1, 2]$.
	\end{thm}
	
	\begin{proof}
		For any $\gamma\in \widehat{K}$, we have
		$$\widehat{\chi}_{_A}(\gamma)= \int_K\overline{\gamma(x)}\chi_{_A}(x) \ d\lambda(x).$$
		If $\gamma\in N(A)$, then $\widehat{\chi}_{_A}(\gamma)=\lambda(A)>0$.
		
		Pick any $a\in A, x\in K$. Note that if $x\in A$, then $\{a\}*\{x\}\subseteq A$. On the other hand, if $\{a\}*\{x\}\subseteq A$, then we have that $x\in \{\widetilde{a}\}*A$ \cite[Lemma 4.1B]{JE}, and hence $x\in A$, since $A$ is a subhypergroup. Thus, $\{a\}*\{x\}\subseteq A$ if and only if $x\in A$. In fact using the same result we can conclude that $x\in \hspace {-0.12in}/\ A = \{\widetilde{a}\}*A$ if and only if $(\{a\}*\{x\}) \cap A = \O$. Hence for any $a\in A$, we have that $L_a\chi_{_A}=\chi_{_A}$.
		
		Now, let $\gamma \in \hspace {-0.11in}/\  N(A)$.Then there exists some $a\in A$ such that $\gamma(a)\neq 1$. Then we have
		\begin{eqnarray*}
			\widehat{\chi}_{_A}(\gamma) &=& \int_K \overline{\gamma(x)} \chi_{_A}(x) \ d\lambda(x)\\
			&=& \int_K \overline{\gamma(a*x)} \chi_{_A}(a*x) \ d\lambda(x)\\
			&=& \overline{\gamma(a)} \int_K \overline{\gamma(x)} \chi_{_A}(x) \ d\lambda(x)\\
			&=& \overline{\gamma(a)} \widehat{\chi}_{_A}(\gamma).
		\end{eqnarray*}
		But we know that $\gamma(a)\neq 1$. Hence we must have that $\widehat{\chi}_{_A}(\gamma)=0$, \textit{i.e,} we have that $\widehat{\chi}_{_A} = \lambda(A)   \chi_{_{N(A)}}$.	Recall that the annihilator of any subset of $K$ is always a closed subset of $\widehat{K}$. In particular, $N(A)$ is a compact open subset of $\widehat{K}$, since $\widehat{\chi}_{_A} \in C_0(\widehat{K})$. For $p=2$, we have that	
		$$||\widehat{\chi}_{_A}||_{_2}^2 = \lambda(A)^2\int_K |\chi_{_{N(A)}}|^2 \ d\pi = \lambda(A)^2\pi(N(A)).$$
		It follows from the Levitan-Plancherel Theorem for hypergroups \cite{BH} that $||\widehat{\chi}_{_A}||_{_2} = ||\chi_{_A}||_{_2}$. Hence we have	
		\begin{eqnarray*}
			\lambda(A)^2\pi(N(A))&=& ||\chi_{_A}||_{_2}^2\ \  = \ \  \lambda(A)\\
			\Rightarrow \pi(N(A)) &=& \lambda(A)^{-1}.
		\end{eqnarray*}
		
		Thus for any $p\in (1, 2)$  we have:
		\begin{eqnarray*}
			||\widehat{\chi}_{_A}||_{p'} &=& \lambda(A)\big{(}\int_K |\chi_{_{N(A)}}|^{p'} \ d\pi\big{)}^{1/{p'}} \\
			&=& \lambda(A) \pi(N(A))^{1/{p'}}\\
			&=& \lambda (A) \lambda(A)^{-1/{p'}}\\
			&=& \lambda(A)^{1/p} = ||\chi_{_A}||_p \ .
		\end{eqnarray*}
		The case for $p=1$ follows immediately from the structure of $\widehat{\chi}_{_A}$.
	\end{proof}

	Now that we have sufficient insight into the structure of a function that attains equality in Hausdorff-Young, we proceed further towards exploring the connection between the  basic uncertainty principle and the best constant achieved in the Hausdorff-Young inequality for $1<p<2$. The following result confirms that if $K\in PW_H^c$ such that $\lambda(Z(K))>0$, then $1$ is the best constant in the Hausdorff-Young inequality  for each $p\in (1, 2)$. The proof of the result is inspired by techniques used in \cite{LV} and \cite{TA}. 
	
	\begin{thm} \label{main3}
		Let $K$ be a commutative hypergroup such that $K\notin PW_H$ and $\lambda(Z(K))>0$. Then $K$ admits a non-trivial function $f\in L^p(K)$ such that $f$ attains equality in Hausdorff-Young, for each $p\in (1, 2)$.
	\end{thm}
	
	In order to prove the above result, keeping Theorem \ref{main1} in mind, we first investigate and look for an open subhypergroup $H$ of $K$ such that the translates of a certain  $f \in C_c(K)$ by a `large' subspace of $M(H)$ cannot  be `too large' in $L^2(K)$ (in the sense of the lemma below). 
	This proves to be fruitful since one can then use the co-efficient functions of the left-regular representation of $K$ to construct a function in $C_0(K)$, which is bounded below on $H$. The following lemma provides such an open subhypergroup for any commutative hypergroup $K$ which does not obey the basic uncertainty principle. This phenomenon is of independent interest as well. 
	
	\begin{lem} \label{lemfd}
		Let $K$ be a commutative hypergroup that admits a non-zero function $f$ on $K$ such that both $f$ and $\widehat{f}$ have compact support in $K$ and $\hk$ respectively. Then the following statements hold true.
		\begin{enumerate}
			\item For any compact neighborhood $C$ of $e$, the linear span of $\{T_\mu f : \mu \in M(C)\}$ is a finite dimensional subspace of $L^2(K)$.
			\item There exists an open subhypergroup $H= \langle U_0\rangle$ of $K$  where $U_0$ is a symmetric open neighborhood of $e$, such that the linear span of $\{T_\mu f : \mu \in M_0(\langle U_0\rangle)\}$ is a finite dimensional subspace of $L^2(K)$. Moreover, we have  $$dim\Big(LS\big(\{T_\mu f : \mu \in M_0(\langle U_0\rangle)\}\big)\Big)= dim \Big(LS\big(\{T_\mu f : \mu \in M(U_0)\}\big)\Big).$$
		\end{enumerate}	
	\end{lem}
	
	\begin{proof}
		Let $C$ be a compact neighborhood of $e$ in $K$. 
		Note that if $x\notin C*  supp(f)$, then in particular for each $y\in C$, we would have that $x\notin \{{y}\}*  supp(f)$ and hence by \cite[Lemma 4.1B]{JE}, $\big{(}\{x\}*\{\widetilde{y}\}\big{)} \cap  supp(f) = \O$, \textit{i.e}, $L_{\widetilde{y}}f(x)=0$. Since any $\mu \in M(C)$ can be perceived as a measure $\mu \in M(K)$ with $supp(\mu)\subseteq C$, we have that $$(\mu *f)(x) = \int_K L_{\widetilde{y}} f(x) \ d\mu(y)=0,$$
		for any $\mu \in M(C)$.	
		Hence for each $\mu \in M(C)$, we have that $supp(T_\mu f) = supp(\mu*f) \subseteq C*  supp(f)$, where $C*  supp(f)$ is compact since both $C$ and  $supp(f)$ are compact \cite{JE}. Now using Urysohn's Lemma, we can find a function $\phi\in C_c(K)$ such that $\phi \equiv 1$ on $C*  supp(f)$, \textit{i.e,}   $\phi  T_\mu f=T_\mu f$ for each $\mu \in M(C)$.
		
		Similarly, since $(T_\mu f)^{\widehat{}}(\gamma) = (\mu * f)^{\widehat{}}(\gamma)=\widehat{\mu}(\gamma)\widehat{f}(\gamma)$ for each $\gamma\in\hk, \mu \in M(K)$, in particular, we have that $supp\big{(}(T_\mu f)^{\widehat{}}\:\big{)} \subseteq supp(\widehat{f})$ for each $\mu \in M(K)$. Since $supp(\widehat{f})$ is compact, using Urysohn's Lemma similarly on $\hk$, we can find a function $\psi\in C_c(\hk)$ such that $\psi(T_\mu f)^{\widehat{}} = (T_\mu f)^{\widehat{}}$ for each $\mu \in M(K)$. Next, define a function $\Omega: L^2(K) \ra L^2(K)$ by
		$$\Omega(h) := \widecheck{\psi}*(\phi h),   $$
		for each $h\in L^2(K)$. First note that since $\psi\in C_c(\hk)\subset L^2(\hk)$,  we have that $\widecheck{\psi} \in L^2(K)$ and hence $\Omega(h) \in L^2(K)$ as $\phi h \in L^1(K)$ for any $h\in L^2(K)$. Thus $\Omega$ is a well-defined bounded linear operator on $L^2(K)$. Consider the function $\rho : K \times K \ra \mathbb{C}$ given by	$$\rho(x, y) := \widecheck{\psi}(x*y) \widetilde{\phi}(y),$$
		for each $(x, y) \in K\times K$. Since $supp(\widetilde{\phi}) = \big{(}supp(\phi){\widetilde{\big{)}}}$ is compact and $\widecheck{\psi}\in C_0(K)\cap L^2(K)$, we have that $\rho\in L^2(K\times K)$. For each $h\in L^2(K), x\in K$, the operator $\Omega$ can be expressed as: 
		$$ \Omega(h)(x) = \int_K \rho(x, y)\widetilde{h}(y) \ d\lambda(y).$$
		Since the operator $i:L^2(K) \ra L^2(K)$ defined as $i(h):= \widetilde{h}$ is a bounded linear operator on $L^2(K)$, it follows immediately that $\Omega$ is a compact operator on $L^2(K)$.	Now for any neighborhood $E$ of $e$ in $K$, denote by $V_E$ the linear subspace of $L^2(K)$ spanned by the set $\{T_\mu f : \mu\in M(E)\}$. In particular, for each $\mu\in M(C)$, we have that
		\begin{eqnarray*}
			\Omega(T_\mu f) &=& \widecheck{\psi}*(\phi T_\mu f)\\
			&=& \widecheck{\psi} * T_\mu f\\
			&=& \big{(} \psi (T_\mu f)^{\widehat{}}  \ \big{)}^{\widecheck{}}\\
			&=& \big{(} (T_\mu f)^{\widehat{}} \ \big{)}^{\widecheck{}} \ = \ T_\mu f,
		\end{eqnarray*}
		where the second and fourth equality follows from the choice of $\phi$ and $\psi$ respectively, and both the third and last equality follows from Remark \ref{remhy}.
		
		Thus the compact operator $\Omega$ coincides with the identity operator on $V_{C}$. Hence the vector space $V_{C}$ must be finite dimensional. For any  neighborhood $W\subseteq C$ of $e$, we have that $V_W\subseteq V_{C}$ and hence $V_W$ has non-zero finite dimension. 
		
		Therefore, we can find a neighborhood $W_0\subseteq C$ of $e$ in  $K$,  such that $dim(V_{W_0})$ is minimal in nature, i.e, $V_{W_0}= V_W $ for any neighborhood $W$ of $e$ such that $W\subseteq W_0$. Since the involution map $x\mapsto \widetilde{x}$ is a homeomorphism, using \cite[Lemma 3.2D]{JE} we can easily find an open neighborhood $U_0$ of $e$ such that $U_0=\widetilde{U}_0$ and $U_0*U_0\subset W_0$. In particular,  by minimality of $W_0$, we have that  $V_{U_0} = V_{W_0}$. Again, for any two measures $\mu, \nu \in M(U_0)$, we have  
		$$ T_\mu(T_\nu f) = T_{\mu*\nu} (f) \in V_{W_0}  =  V_{U_0},$$
		as both $supp(\mu)$ and $supp(\nu)$ are compact in $K$ and hence $\mu *\nu \in M(W_0)$ since $supp(\mu *\nu) = supp(\mu) * supp(\nu) \subseteq U_0 * U_0 \subset W_0$. 
		Thus for any $\mu \in M(U_0)$, we have that $T_\mu(V_{U_0}) \subseteq V_{U_0}$. 
		
		Now, let $H:= \langle U_0 \rangle$. First note that since $H\supset U_0$  has non-empty interior, $H$ is an open subhypergroup of $K$ \cite{BH}.  Moreover, let $\mu= \mu_1 * \mu_2 *\cdots *\mu_n$ for some $n\in \mathbb{N}$ where $\mu_i\in M(U_0)$ for $1\leq i \leq n$. Using the invariance derived above, we have 
		\begin{eqnarray*}
			T_\mu f &=& T_{\mu_1 * \mu_2 *\cdots *\mu_n} (f)\\
			&=&  T_{\mu_1}\circ  T_{\mu_2} \circ \ldots \circ (T_{\mu_{n-1}}\circ T_{\mu_n})(f)\\
			&\in & T_{\mu_1}\circ  T_{\mu_2} \circ \ldots \circ T_{\mu_{n-2}}\big{(} V_{U_0}\big{)}\\
			&\subseteq & T_{\mu_1}\circ  T_{\mu_2} \circ \ldots \circ T_{\mu_{n-3}}\big{(} V_{U_0}\big{)} \subseteq \cdots \subseteq T_{\mu_1}(V_{U_0}) \subseteq V_{U_0}.
		\end{eqnarray*}
		Thus we have that $LS\big(\{T_\mu f : \mu \in M_0(\langle U_0\rangle)\}\big) \subseteq V_{U_0}$. Since $M(U_0)\subseteq M_0(\langle U_0\rangle)$, we immediately have that $LS\big(\{T_\mu f : \mu \in M_0(\langle U_0\rangle)\}\big)= V_{U_0} = V_{W_0}$, and therefore $$0<dim\Big(LS\big(\{T_\mu f : \mu \in M_0(\langle U_0\rangle)\}\big)\Big)= dim \big(V_{U_0}\big)<\infty.$$ 
	\end{proof}
	The above lemma immediately implies that if $K\notin PW_H$ and $f\in C_c(K)$ is a non-zero function such that $\widehat{f}\in C_c(\hk)$ as well, then there exists an open subhypergroup $H=\langle U_0\rangle$ of $K$ such that the linear span of $\{T_\mu f : \mu \in M_0(H)\}$ coincides with the linear span of $\{L_xf : x\in U_0\}$, which is finite dimensional. Now we return to the proof of \Cref{main3} as discussed before.

	\noi
	\textbf{Proof of \Cref{main3}:}

	
	We assume the same notations here as in the proof of Lemma \ref{lemfd}. Since $K\notin PW_H$, there exists a non-zero function $f\in C_c(K)$ such that $\widehat{f}\in C_c(\hk)$. Hence given any compact neighborhood $C$ of $e$ in $K$, using \Cref{lemfd} we can find a symmetric open neighborhood $U_0\subset C$ of $e$ in $K$ such that $H= \langle U_0 \rangle$ is an open subhypergroup of $K$, and the linear span of $\{T_\mu f : \mu \in M_0(H)\}$ is finite dimensional.  Set $V_0:=LS\big{(} \{T_\mu f : \mu \in M_0(H)\} \big{)}$, and let $d= dim(V_0)=dim(V_{U_0})$.

	Now, let $\{e_i\}_{i\in I}$ be an orthonormal basis for $L^2(K)$ such that there exists $n_i\in \mathbb{N}$, $1\leq i\leq d$ such that $\{e_{n_i} : 1\leq i \leq d\}$ is a basis for $V_0$. For each $i,j \in \{1, 2, \ldots, d\}$ we define the co-efficient function $a_{i,j} : M(K) \ra \mathbb{C}$ as
	$$a_{i, j} (\mu) := \langle T_\mu e_{n_i}, e_{n_j}\rangle$$
	for each $\mu\in M(K)$. Hence in particular, since $T_\mu(e_{n_k})\in V_0$ for each $\mu\in M_0(H)$, $1\leq k \leq d$, we have that $$T_\mu(e_{n_k}) = \sum_{i=1}^d \langle T_\mu e_{n_k}, e_{n_i}\rangle e_{n_i} = \sum_{i=1}^d a_{k,i}(\mu) e_{n_i}. $$
	For each $\mu, \nu\in M_0(H)$, $1\leq i, j \leq d$, we further have the following equality:
	\begin{eqnarray*}
		a_{i,j}(\mu*\nu) &=& \langle T_{\mu * \nu} (e_{n_i}), e_{n_j}\rangle\\
		&=& \langle T_{\mu} \big{(}T_{\nu} (e_{n_i})\big{)}, e_{n_j}\rangle\\
		&=& \Big{\langle} T_\mu \Big{(} \sum_{k=1}^d a_{i,k}(\nu) e_{n_k}\Big{)}, e_{n_j}\Big{\rangle}\\
		&=&  \sum_{k=1}^d a_{i,k}(\nu) \Big{\langle} T_\mu {(}   e_{n_k} {)}, e_{n_j}\Big{\rangle}\\
		&=&  \sum_{k=1}^d a_{i,k}(\nu) \Big{\langle}\sum_{l=1}^d a_{k, l}(\mu)e_{n_l} , e_{n_j}\Big{\rangle}\\	
		&=& \sum_{k=1}^d \sum_{l=1}^d  a_{i,k}(\nu)a_{k, l}(\mu) {\langle}   e_{n_l} , e_{n_j}{\rangle} = \sum_{k=1}^d a_{i,k}(\nu)a_{k, j}(\mu).
	\end{eqnarray*}
	Now define a function $F: K \ra \mathbb{C}$ as $F(x) :=  \sum_{k=1}^d a_{1,k}(p_{_x})a_{k, 1}(p_{_{\widetilde{x}}})$
	for each $x\in K$. Note that since the map $y \mapsto \langle T_{y}g_1, g_2 \rangle : K \ra \mathbb{C}$ is in $C_0(K)$ for any two functions $g_1, g_2 \in L^2(K)$, we have that each $a_{i,j} \in C_0(K)$ for $1\leq i, j \leq d$. Hence in particular, $F\in C_0(K)$. 

	But we also have that $\lambda(Z(K))>0$. Hence there exists a compact subset $C$ in $Z(K)$ such that $e\in C$ and $0<\lambda(C)<\infty$. It now follows from the Steinhaus' Theorem for hypergroups \cite[Proposition 1.3.30]{BH} that there exists an open neighborhood $V$ of $e$ such that $V\subset C \subseteq Z(K)$, \textit{i.e,}  $C$ is a compact neighborhood of $e$ in $K$. Hence in particular,  Lemma \ref{lemfd} implies that $U_0 \subset C \subset Z(K)$ and hence by \Cref{subgen} we have that $H\subseteq Z(K)$. Thus for each $x\in H$, $T_x$ is an isometry on $L^2(K)$ and $p_{_x}\in M_0(H)$. Thus $F\equiv 1$ on $H$ since for each $x\in H$ we have:
	\begin{eqnarray*}
		F(x) = a_{1,1}(p_x * p_{\widetilde{x}}) &=& \langle T_{p_{_{\widetilde{x}}}*p_{_x}} e_{n_1},  e_{n_1}\rangle\\
		&=& ||T_x e_{n_1}||_{_2}^2 = ||e_{n_1}||_{_2}^2 = 1.
	\end{eqnarray*}
	Hence $H$ must be compact. It now follows immediately from Theorem \ref{main1} that $\chi_{_H}$ attains equality in Hausdorff-Young on $L^p(K)$ for each $p\in (1, 2)$.
	\qed
	
	\begin{rem}
		In the first half of the proof of the above theorem, we consider a general commutative hypergroup in $PW_H^c$, without any restriction on the centre $Z(K)$, and acquire a certain function in $C_0(K)$. 
		However, the lack of an isometry in terms of translation operators leads to the phenomenon that for a non-zero function $f\in C_c(K)$, one fails to directly provide a constant $C>0$ such that $\liminf_{x\in H}||L_xf||_{_2}\geq C$ where $H$ is a subhypergroup of $K$. Note that the proof becomes substantially simpler if we assume a priori that $\lambda(Z(K))>0$, and thereby choose $U_0\subseteq H\subseteq Z(K)$. However, the above version of the proof provides a clearer insight into the necessity of such restriction.
		
		Observe that for any hypergroup $K$, the centre $Z(K)$ is the largest locally compact group contained in $K$. In particular, the restriction $\lambda(Z(K))>0$ is equivalent to the fact that $Z(K)$ is open as well. Therefore if $\lambda(Z(K))>0$, then  $K$  is a disconnected space, unless $K=Z(K)$, i.e, $K$ is simply a locally compact group. For non-trivial examples of such occurrence, note that if $K$ is any discrete hypergroup, then $\{e\}$ is an isolated point, and hence $\lambda(Z(K))\geq \lambda(\{e\})>0$. Moreover, if $K$ is a hypergroup join $(H\wedge J)$ where $H$ is a compact group, 
		then we always have that $\lambda(Z(K))= \lambda(H)>0$. Similarly, for any two hypergroups $K_1, K_2$ such that $\lambda_{K_i}(Z(K_i))>0$ for $i=1, 2$, we have that $Z(K_1\times K_2)$ has non-zero measure.
		
		On the other hand, every disconnected hypergroup $K$ need not satisfy that $\lambda(Z(K))>0$. We briefly mention below a few of such examples (see \cite[6.13, 6.14]{RO} for details).
		\begin{ex} \label{exx}
			Consider the hypergroup $K:=(H_0\wedge J_0)$, where $H_0=([0, 1], *_1)$ given by $p_{_x} *_{_1} p_{_y}:= \frac{1}{2}p_{_{|x-y|}} + \frac{1}{2}p_{_{(1-|1-x-y|)}} $, and $J=(\{0, a\}, *_2)$ with identity $0$ given by $p_{_a} *_{_2} p_{_a} = \frac{1}{3}p_{_0} + \frac{2}{3}p_{_a}$. Then clearly, $K$ is disconnected and $Z(K)= \{0, 1\}$, which has empty interior in the locally compact space $[0,1]\cup\{a\}$, and hence has Haar-measure zero.
			
			\noi In fact, if $K=(H\wedge J)$ where $\lambda_H(Z(H))=0$  and both $H, J$ are non-trivial, then $\lambda(Z(K))=0$, although the hypergroup $K=(H\wedge J)$ is disconnected.
		\end{ex}

		\begin{ex}
			For any prime number $p$, let $\mathtt{Z}_p$ denote the compact group of $p$-adic integers, and let $\mathcal{B}\subseteq Aut(\mathtt{Z}_p)$ be the group of units in $\mathtt{Z}_p$ acting multiplicatively on $\mathtt{Z}_p$. We know that the set consisting of orbits of  $\mathtt{Z}_p$ under $\mathcal{B}$ has a natural hypergroup structure \cite[Section 8]{JE}. In fact, the resulting hypergroup $K$ can be identified with the one-point compactification $\{0, 1, 2, \ldots, \infty\}$ of $\mathbb{Z}^+$, which is totally disconnected. However, we have that $Z(K)=\{\infty\}$, which has $\lambda$-measure zero, since $K$ is not discrete \cite[Theorem 1.3.27]{BH}.
		\end{ex}
		
		\begin{ex}
			For a non-compact example, simply consider the join $K= H_0\wedge \mathbb{Z}$ or  the product\cite{JE} $K=H\times J_0$ where $H=(\mathbb{R}^+, *_H)$; $p_{_x} *_{_H} p_{_y}:= \frac{1}{2}p_{_{|x-y|}} + \frac{1}{2}p_{_{(x+y)}} $, for each $x, y \in \mathbb{R}^+$. Then $K$ is disconnected, and $Z(K)=\{(0, 0)\}$, which has $\lambda$-measure zero.
		\end{ex}

	\end{rem}
	
	In particular, as a corollary we acquire the following characterization of certain classes of commutative hypergroups 
	which follows the  basic uncertainty principle, in terms of the best constant achieved in the Hausdorff-Young inequality.
	
	\begin{cor}\label{fcor}
		Let $K$ be a commutative hypergroup such that either $K$ is compact or $\lambda(Z(K))>0$. Then $K\notin PW_H$ if and only if $K$ admits a non-trivial function $f\in L^p(K)$ such that $f$ attains equality in Hausdorff-Young, for some $p\in (1, 2)$.
	\end{cor}
	
	\begin{proof}
		It follows immediately from Theorem \ref{main2} that if $f$ is a non-zero function in $L^p(K)$ such that $||\widehat{f}||_{p'} = ||f||_p$ for some $p\in (1, 2)$, then $K\in PW_H^c$. On the other hand, let $K\notin PW_H$. Then the result follows immediately from Theorem \ref{main3} if $\lambda(Z(K))>0$ and from Theorem \ref{main1} if $K$ is compact, and hence itself is a compact open subhypergroup.
	\end{proof}
	
	\begin{rem}
		Although there are examples of non-compact commutative hypergroups $K$ such that $\lambda(Z(K))=0$, where $K\notin PW_H$ and $K$ admits non-trivial $f\in L^p(K)$ that attains equality in Hausdorff-Young for each $1\leq p \leq 2$ as well (Consider $K=H_0\wedge \mathbb{Z}$ with $\chi_{_{[0, 1]}}$), it is not known yet if for any non-compact commutative hypergroup $K$ with $\lambda(Z(K))=0$, the equivalence in Corollary \ref{fcor} will be satisfied.
		
		
	\end{rem}
	
	
	\section*{Acknowledgment}
	The authors are grateful to the referee for the thoughtful remarks and suggestions, which led to overall improvement of the article. The first named author would like to gratefully acknowledge the financial support provided by the Indian Institute of Technology Kanpur, India, throughout the course of this research.


\begin{thebibliography}{1}
		\bibitem{B}{Babenko, K.I.  \textit{An inequality in the theory of Fourier integrals.} Izv. Akad. Nauk SSSR Ser. Mat. 25, 531-542 (1961) (Russian); translated as Am. Math. Soc. Transl. Ser. 2 44:115-128 (1962)}
		
		
		\bibitem{CB1} {Bandyopadhyay, Choiti \textit{Analysis on semihypergroups: function spaces, homomorphisms and ideals}. Semigroup Forum 100 (2020), no. 3, 671-697.}
		
		\bibitem{CB2} {Bandyopadhyay, Choiti \textit{Free Product of Semihypergroups}. Semigroup Forum 102 (2021), no. 1, 28–47.}
		
		\bibitem{Be} {Beckner, W. \textit{ Inequalities in Fourier analysis} . Ann. Math. (2) 102, 159-182 (1975)}
		
		\bibitem{BH}{Bloom, Walter R.; Heyer, Herbert \textit{Harmonic analysis of probability measures on hypergroups}. De Gruyter Studies in Mathematics, 20. Walter de Gruyter \& Co., Berlin, 1995.}
		
		\bibitem{CMP} {Cowling, Michael G.,  Martini, A., M\"uller, Detlef, Parcet, Javier \textit{The Hausdorff-Young inequality on Lie groups}. Math. Ann. 375 (2019), no. 1-2, 93-131.}
		
		
		\bibitem{SD}{Degenfeld-Schonburg, Sina \textit{On the Hausdorff-Young theorem for commutative hypergroups}. Colloq. Math. 131 (2013), no. 2, 219-231.}
		
		
		\bibitem{HH}{Hewitt, Edwin; Hirschman, Isidore, Jr. \textit{A maximum problem in harmonic analysis}. Amer. J. Math. 76 (1954), 839-852.}
		
		\bibitem{HR2}{Hewitt, Edwin; Ross, Kenneth A. \textit{Abstract harmonic analysis. Vol. II: Structure and analysis for compact groups. Analysis on locally compact Abelian groups}. Die Grundlehren der mathematischen Wissenschaften, Band 152 Springer-Verlag, New York-Berlin 1970.}
		
		
		\bibitem{H2} {Hirschman, Isidore, Jr. \textit{ A maximal problem in harmonic analysis. II. }Pacific J. Math. 9 (1959), 525-540.}
		
		
		\bibitem{JE}{Jewett, Robert I. \textit{Spaces with an abstract convolution of measures}. Advances in Math. 18 (1975), no. 1, 1-101.}
		
		
		\bibitem{LV}{Landstad, Magnus B.; Van Daele, A. \textit{Groups with compact open subgroups and multiplier Hopf $*$-algebras}. Expo. Math. 26 (2008), no. 3, 197-217.}
		
		\bibitem{MT} {Michael, Ernest \textit{Topologies on spaces of subsets}. Trans. Amer. Math. Soc. 71 (1951), 152-182.}
		
		\bibitem{RO}{Ross, Kenneth A. \textit{Centres of Hypergroups.} Trans. Amer. Math. Soc. 243 (1978), 251-269. }
		
		
		
		\bibitem{R}{Russo, B. \textit{The norm of the $L^p$-Fourier transform on unimodular groups.} Trans. Amer. Math. Soc. 192 (1974), 293-305. }
		
		\bibitem{TA}{Takai, Hiroshi \textit{On a duality for crossed products of C*-algebras.} J. Functional Analysis 19 (1975), 25–39.}
		
		
	\end{thebibliography}
\end{document}